\documentclass{article}
\usepackage{arxiv}

\usepackage[normalem]{ulem} %to strike the words

\usepackage[utf8]{inputenc} % allow utf-8 input
\usepackage[T1]{fontenc}    % use 8-bit T1 fonts
\usepackage{hyperref}       % hyperlinks
\usepackage{url}            % simple URL typesetting
\usepackage{booktabs}       % professional-quality tables
\usepackage{amsfonts}  
\usepackage{amssymb}
\usepackage{color}
\usepackage{lipsum}
\usepackage{amsthm}
\usepackage{indentfirst}
\usepackage{graphicx}
\usepackage{epsfig,epsf,latexsym,subfigure}
\usepackage{float}
\usepackage{epstopdf}
\usepackage{fourier}
\usepackage{bm}
\usepackage{bbm} % for Indicator function
\usepackage{mathtools}
\usepackage{latexsym,enumerate}
\usepackage{tabularx}
\usepackage{todonotes}

\usepackage{capt-of}
\usepackage{cancel}

\usepackage{cases}
\usepackage{enumitem} 

\newcommand{\comment}[1]{}

\newcommand{\BEA}{\begin{eqnarray}}
\newcommand{\EEA}{\end{eqnarray}}

\newcommand{\BR}{\mathbb{R}}

\newtheorem{lem}{Lemma}[section]
\newtheorem{thm}{Theorem}[section]

\newtheorem{remark}{Remark}[section]

\title{A Geometric Local Parameterization Method for Generalized Hele–Shaw Free Boundary Problems with Source Terms}

\author{
  Zengyan Zhang and Wenrui Hao \\
  Department of Mathematics \\
  The Pennsylvania State University, University Park, PA 16802, USA\\
  \texttt{zzz5527@psu.edu} and \texttt{wxh64@psu.edu} \\
  \And
 John Harlim \\
  Department of Mathematics, \\ Institute for Computational and Data Sciences \\
  The Pennsylvania State University, University Park, PA 16802, USA\\
  \texttt{jharlim@psu.edu} \\
}

\begin{document}

\maketitle

\begin{abstract}

We develop a meshfree numerical framework for Hele--Shaw free boundary problems with surface tension and source terms based on geometric local parameterization and boundary integral methods. By decomposing the pressure into a particular solution and a harmonic component, the problem is reformulated into a boundary-only system, avoiding volumetric meshing of the evolving domain. For general source terms, we propose an eigenfunction-based approximation on a fixed domain and establish error estimates for both the truncation and coefficient approximation. Numerical experiments verify the accuracy and convergence of the proposed method, and an application to a tumor growth model demonstrates its effectiveness for coupled moving-boundary problems.

\end{abstract}

\section{Introduction}

Free boundary problems arise in a broad range of scientific applications, including fluid dynamics, porous media flow, and tumor growth \cite{friedman2000free,friedman2015free,HHHS}. Among them, the generalized Hele-Shaw problem with surface tension serves as a fundamental model for interface motion driven by pressure gradients under the influence of surface tension \cite{chen2003free,friedman2012variational,richardson1972hele}. This problem has been extensively studied since the seminal work of Saffman and Taylor \cite{saffman1958penetration}, following Hele-Shaw's pioneering experimental studies of viscous flow between two closely spaced plates \cite{HSH}. In its classical form, the pressure satisfies Laplace's equation in the evolving domain, and the interface velocity is determined by Darcy's law through the pressure gradient. Owing to the coupling between the pressure field and the evolving interface, the development of accurate and efficient numerical methods for Hele-Shaw problems remains a challenging task\cite{tryggvason1983numerical}.

A variety of numerical methods have been developed for Hele-Shaw free boundary problems and their applications, including finite difference methods, finite element methods, level-set methods, and phase field methods \cite{feng2004analysis,macklin2007nonlinear,ryskin1984numerical,sackinger1996newton}. Although these approaches have been successfully applied to a wide range of free boundary problems, they generally require discretization of the computational domain together with repeated interface reconstruction or mesh generation as the free boundary evolves, leading to increased computational cost.

Boundary integral methods provide an attractive alternative for Hele-Shaw problems because they exploit the Green's function of the governing equation to reduce the problem to the evolving interface, thereby avoiding volumetric discretization and reducing the dimensionality of the computation by one \cite{hou1994removing,pozrikidis1992boundary}. Moreover, rigorous convergence studies of boundary integral methods for free boundary problems have been established \cite{HHLS}, providing a solid theoretical foundation for their application. For example, we recently introduced a geometric local parameterization method based on generalized moving least squares (GMLS) for solving classical Hele–Shaw problems directly on point clouds \cite{zhang2026geometric}. By reconstructing local manifold parameterizations from unorganized boundary points, the method accurately approximates geometric quantities and boundary integrals without requiring global parameterization or remeshing.

Nevertheless, most existing formulations (including \cite{zhang2026geometric}) are limited to homogeneous equations or rely on parameterized interfaces and body-fitted meshes. When nonzero source terms are present, as in tumor growth and reactive porous media models, the resulting boundary integral formulation involves volume integrals over the evolving domain. Directly approximating the volume integral arising from the nontrivial source term eliminates the dimensionality reduction advantage offered by the boundary integral formulation. Furthermore, evaluating the time dependent volume integral term generally requires repeated interior discretization or mesh generation when the geometry is complex, resulting in significant computational cost.

In this work, we develop a numerical framework for generalized Hele–Shaw free boundary problems with source terms. By decomposing the pressure into a particular solution and a harmonic component, we reformulate the original problem as a boundary-only system, thereby eliminating repeated volume integration over evolving domains. For source terms whose particular solutions are not available in closed form, we propose an eigenfunction approximation based on Laplacian eigenfunctions defined on a fixed domain \cite{courant2024methods,shen1995efficient}. The associated Poisson problem is solved offline, enabling the subsequent free-boundary evolution of the harmonic component to be computed entirely through boundary integrals and solved using the geometric local parameterization method that we recently developed \cite{zhang2026geometric}. We further establish error estimates for the eigenfunction approximation and the numerical evaluation of the expansion coefficients. A series of numerical experiments validates the proposed framework through spatial and temporal convergence studies, demonstrates the convergence behavior of the eigenfunction approximation, and confirms that the pressure approximation converges at a higher rate than the source approximation. Finally, we apply the proposed method to a nutrient-driven tumor growth model, illustrating its capability for solving coupled moving-boundary problems with source terms without repeated volumetric remeshing.

The remainder of the paper is organized as follows. In Section \ref{sec:overview}, we briefly review the geometric local parameterization method developed in \cite{zhang2026geometric} for solving the Hele-Shaw free boundary problems with zero source terms. Section \ref{sec:non-zero-source} presents the numerical method for generalized Hele-Shaw problems with non-zero source terms. Section \ref{sec:eigenfunction-approx} introduces the eigenfunction approximation for the source term and establishes the corresponding error estimates. Finally, Section \ref{sec:applications} presents numerical results, including applications to a tumor growth model and extensions to more general source terms.

\section{Overview of the Hele-Shaw problem with surface tension and geometric local parameterization method}\label{sec:overview}

Consider a tissue or material that behaves as a porous medium. The velocity field $\mathbf V$ is described by Darcy's law, 
$\mathbf{V} = -\frac{\sigma}{\mu} \nabla p,$
where $p$ denotes the pressure, $\sigma$ is the permeability of the medium, and $\mu$ is the dynamic viscosity of the fluid. Combining with the conservation of mass, $\nabla \cdot \mathbf V=f$, where $f:\Omega(t) \to \BR$ represents distributed sources or sinks, leads to the following PDE system with a free boundary,
\begin{equation}\label{hele-shaw}
\left\{
\begin{array}{rcll}
-\frac{\sigma}{\mu}\Delta p &=& f & \text{in }\, \Omega(t), \\
p &=& \tau\kappa & \text{on }\, \Gamma(t), \\
\frac{\sigma}{\mu}\frac{\partial p}{\partial \mathbf n} &=& -V_n & \text{on }\, \Gamma(t).
\end{array}
\right.
\end{equation}
Here, \( \Omega(t) \) denotes the fluid-filled region, and \( \Gamma(t) \) denotes its moving boundary. In this paper, we consider a bounded two-dimensional domain $\Omega(t) \subset \mathbb{R}^2$ with a closed curve boundary $\Gamma(t)$ at any fixed time $t\geq 0$. In \eqref{hele-shaw}, \( \kappa \) denotes the mean curvature (for instance, $\kappa=r(t)^{-1}$ if $\Omega(t)$ is a disk of radius $r(t)$).  
The boundary condition $p=\tau\kappa$ models the effect of surface tension: the pressure jump across the interface is proportional to the local curvature, which acts to stabilize and regularize the evolving boundary. In classical Hele–Shaw problems without surface tension, the pressure on the boundary is typically prescribed as a constant or a given potential function \cite{HHLS}. The inclusion of surface tension distinguishes the present model, as the boundary pressure is not known a priori but instead couples dynamically to the evolving geometry through curvature.

For simplicity, we normalize the parameters by setting $\sigma=\mu$ and $\tau=1$ throughout the remainder of the paper.
We assume the kinematic boundary condition on the free boundary $\Gamma(t)$, which states that the boundary moves in accordance with the velocity $\mathbf V$, such that its component in the direction of the outward normal \(\mathbf n \) is given by
\begin{equation*}
V_n=\mathbf V\cdot\mathbf n=-\frac{\partial p}{\partial \mathbf n} \quad\text{on}~\Gamma(t),
\end{equation*}
with $\mathbf n$ pointing out of $\Omega(t)$. 

\subsection{Boundary integral formulation}
Denote $G(\mathbf x,\mathbf y)$ as the Green's function for the Laplacian operator, $-\Delta$. Specifically, \( G(\mathbf{x}, \mathbf{y}) = -\frac{1}{2\pi} \ln\|\mathbf{x} - \mathbf{y}\| \) for two-dimensional cases. Then from the Green's third identity \cite{friedman1963generalized,kress1989linear}, we have
\[-\int_\Omega G(\mathbf x,\mathbf y)\Delta p(\mathbf y) dV_{\mathbf y}-p(\mathbf x)=-\int_{\Gamma}\Big(G(\mathbf x,\mathbf y)\frac{\partial p(\mathbf y)}{\partial\mathbf n(\mathbf y)}-p(\mathbf y)\frac{\partial G(\mathbf x,\mathbf y)}{\partial\mathbf n(\mathbf y)}\Big) dS_{\mathbf y}, \quad\forall \mathbf x\in\Omega(t) \]

Setting $-\Delta p=f$, incorporating the jump condition and the boundary condition, we arrive at 
\[
\frac{\kappa(\mathbf{x})}{2} - \int_{\Omega} G(\mathbf x,\mathbf y)f(\mathbf y)dV_{\mathbf y}=\int_{\Gamma}\Big(-G(\mathbf{x},\mathbf{y})V_n(\mathbf y) -\kappa(\mathbf{y})\frac{\partial G(\mathbf{x},\mathbf{y})}{\partial \mathbf{n}(\mathbf y)}\Big)dS_\mathbf{y},\quad \mathbf{x},\mathbf{y} \in \Gamma(t),
\] 
where $\mathbf{n}=\frac{\nabla \Gamma}{|\nabla \Gamma|}$ with $\mathbf n(\mathbf y)$ referring to the unit outer normal vector at $\mathbf y\in\Gamma(t)$, and $\kappa=\nabla\cdot\mathbf{n}$.

Therefore, we have reformulated the Hele-Shaw free boundary problem \eqref{hele-shaw} as the following system,
\begin{subnumcases}{\label{BIM}}
\int_{\Gamma}G(\mathbf{x},\mathbf{y})V_n(\mathbf y) dS_\mathbf{y}=-\frac{\kappa(\mathbf{x})}{2} -\int_\Gamma \kappa(\mathbf{y})\frac{\partial G(\mathbf{x},\mathbf{y})}{\partial \mathbf{n}(\mathbf y)}dS_\mathbf{y} + \int_{\Omega}G(\mathbf x,\mathbf y)f(\mathbf y)dV_{\mathbf y}  ,\quad \mathbf{x},\mathbf{y} \in \Gamma(t),\label{BIE}\\
\frac{d\mathbf x}{dt} =  V_n(\mathbf x) \mathbf n(\mathbf x), \quad \mathbf x\in\Gamma(t),\label{velocity}
\end{subnumcases}
where the solution to \eqref{BIE} gives the unknown normal velocity $V_n$ on $\Gamma(t)$, which is represented by point clouds without parameterization information, and the kinematic condition \eqref{velocity} enables the free boundary to evolve forward in time. In particular, when $f=0$, corresponding to the classical Hele-Shaw free boundary problem, \eqref{BIE} reduces to a boundary integral equation (BIE), allowing the governing equations to be formulated solely on the free boundary. In the following section, we briefly summarize our previous work \cite{zhang2026geometric} that handles the homogeneous case $f=0$, which serves as the foundation for the development of the present work on problems with nontrivial source terms $f\neq 0$.

\subsection{Review of geometric local parameterization method}\label{sec:GMLS}
Now, we briefly summarize the numerical framework developed in our previous work \cite{zhang2026geometric} for the classical Hele-Shaw free boundary problem. With $f(\mathbf x)=0$, the system in \eqref{BIE}-\eqref{velocity} reduces to the following system,
\begin{subnumcases}{\label{BIM1}}
\int_{\Gamma}G(\mathbf{x},\mathbf{y})V_n(\mathbf y) dS_\mathbf{y}=-\frac{\kappa(\mathbf{x})}{2} -\int_\Gamma \kappa(\mathbf{y})\frac{\partial G(\mathbf{x},\mathbf{y})}{\partial \mathbf{n}(\mathbf y)}dS_\mathbf{y}  ,\quad \mathbf{x},\mathbf{y} \in \Gamma(t),\label{BIE1}\\
\frac{d\mathbf x}{dt} =  V_n(\mathbf x) \mathbf n(\mathbf x), \quad \mathbf x\in\Gamma(t),\label{velocity1}
\end{subnumcases}
where $\Gamma$ is a free boundary that will be characterized by $\mathbf x\in \Gamma(t)$. Throughout the discussion below, we assume that $\Gamma$ is smooth.

\subsubsection{Generalized Moving Least Squares approximation of the manifold parameterization}
First, we construct a local representation of the manifold through the Generalized Moving Least Squares (GMLS) method \cite{jiang2024generalized,mirzaei2012generalized,zhang2026geometric}. This approximation allows one to estimate higher-order geometric quantities, such as the mean curvature $\kappa$ along moving boundaries and naturally allows one to approximate the integration on the boundary directly. 

 Given a set of point cloud data $X=\{\mathbf x_i\}_{i=1}^N\subset\mathbb{R}^n$ sampled from a manifold $\Gamma$ of dimension $d$, let $K(i)=\{\mathbf x_{i_1},\dots,\mathbf x_{i_k}\}$ with $\mathbf x_i=\mathbf x_{i_1}$ denote the set of $k$-nearest neighbors for any point $\mathbf x_i$ in $X$. In this work, we focus on the case $d=1$ and $n=2$. The GMLS algorithm approximates a local parameterization of $\Gamma$ around $\mathbf x_i$ by the embedding map, $\hat{\iota}_i:T_{\mathbf{x}_i}\Gamma \equiv\mathbb{R} \to  \Gamma \subset\BR^2$, defined by
\begin{equation}\label{gmls-map}
\hat{\iota}_i(s)=\mathbf x_i+\mathbf{\hat{t}}_is+\mathbf{\hat{n}}_i p_i(s),
\end{equation}
where $\hat{\mathbf{t}}_i$ and $\hat{\mathbf n}_i$ are estimates of tangent and normal vectors at $\mathbf x_i$, respectively. See Appendix~\ref{appx:GMLS} for details on how to construct these basis vectors. Also, $p_i(s)$ is an intrinsic polynomial of degree $\ell$,
\[p_i(s)=\alpha_{i,1}s+\alpha_{i,2} s^2+\cdots+\alpha_{i,\ell} s^\ell\] 
whose coefficients are determined via a local least-squares fit using the data pairs from $T(i)=\{t_1=\mathbf{\hat{t}}_i\cdot(\mathbf x_{i_1}-\mathbf x_{i_1}),\dots,t_k=\mathbf{\hat{t}}_i\cdot(\mathbf x_{i_k}-\mathbf x_{i_1})\}\subset T_{\mathbf x_i}\Gamma$ and $N(i) =\{n_1=\mathbf{\hat{n}}_i\cdot(\mathbf x_{i_1}-\mathbf x_{i_1}),\dots,n_k=\mathbf{\hat{n}}_i\cdot(\mathbf x_{i_k}-\mathbf x_{i_1})\}$, 
\[(\alpha_{i,1},\dots,\alpha_{i,\ell})=\arg\min \sum_{j=1}^k\Big(p_i(t_j)-n_j\Big)^2.
\]

Let $\gamma_i(s)=(s,p_i(s))$ be the coordinates in \eqref{gmls-map}, which is a local parametric representation of $\Gamma(t)$ near the point $\mathbf x_i$. The curvature near $\mathbf x_i$ is then approximated by,
\begin{equation}\label{cur_gmls}
\kappa_i(s) =-\frac{\text{det}(\gamma',\gamma'')}{||\gamma'||^3}= -\frac{p_i''(s)}{\Big(1+\big(p_i'(s)\big)^2\Big)^{\frac{3}{2}}},
\end{equation}
with the estimated curvature at $\mathbf x_i$ denoted by $\kappa_i=\kappa_i(0)$.

\subsubsection{Spatial and temporal discretization}
Once the GMLS approximation of the local manifold parameterization and curvature have been constructed, we can discretize the BIE \eqref{BIE} in space, which can be written as a linear system $AV_n(\mathbf{x})=b(\mathbf{x})$, with 
\begin{eqnarray}
A V_n(\mathbf{x}) &=& \int_{\Gamma} G(\mathbf x,\mathbf y) V_n(\mathbf y) dS_{\mathbf y}, \label{eq:AVn} \\ b (\mathbf{x})&=&-\frac{\kappa(\mathbf x)}{2}-\int_\Gamma \kappa(\mathbf y)\nabla G(\mathbf x,\mathbf y)\cdot \mathbf{n}(\mathbf y)dS_{\mathbf y}, \label{eq:b}
\end{eqnarray}
where $G(\mathbf x,\mathbf y)=-\frac{1}{2\pi}\ln\|\mathbf x-\mathbf y\|$ and $\nabla G(\mathbf x,\mathbf y) = \frac{1}{2\pi}\frac{\mathbf x-\mathbf y}{\|\mathbf x-\mathbf y\|^2}$.

Locally, near $\mathbf{x}_j$, we approximate
\BEA dS_{\mathbf y} = \sqrt{|\iota_j'(s)|}ds \approx \sqrt{|\hat{\iota}_j'(s)|}ds = \sqrt{1+\big(p_j'(s)\big)^2}ds, \label{eq:approxJac}\EEA
where we use the local parameterization $\mathbf y=\hat{\iota}_j(s)$ in \eqref{gmls-map} near $\mathbf{x}_j$ for all $j=1,\ldots, N$. In terms of the local coordinate system through a change of variables, at $\mathbf x=\mathbf x_i\in\Gamma(t)$,
\BEA\label{eqn:A-int}
AV_n(\mathbf{x}_i)&\approx& \int_{\Delta s_{-i}}^{\Delta s_i} G\big(\mathbf x_i,\hat{\iota}_i(s)\big)V_n\big(\hat{\iota}_i(s)\big)\sqrt{1+\big(p_i'(s)\big)^2}ds 
\notag \\ &&+\sum_{j\neq i-1,i}\int_0^{\Delta s_j} G\big(\mathbf x_i,\hat{\iota}_j(s)\big)V_n\big(\hat{\iota}_j(s)\big)\sqrt{1+\big(p_j'(s)\big)^2}ds,\label{eq:AVdetail}
\EEA
 where $\Delta s_i={\hat{\mathbf{t}}_i}^\top(\mathbf x_{i+1}-\mathbf x_i)>0$, $\Delta s_{-i}={\hat{\mathbf{t}}_i}^\top(\mathbf x_{i-1}-\mathbf x_{i})<0$. 

Next, we formulate $b(\mathbf x_i)$ with the local parameterized volume in \eqref{eq:approxJac} and write the integral in terms of the local coordinate system through a change of variables,
\begin{equation}\label{eqn:b_int}
\begin{aligned}
b(\mathbf x_i)\approx-\frac{\kappa(\mathbf x_i)}{2}&-\frac{1}{2}\Bigg[\int_{\Delta s_{-i}}^{\Delta s_i} \kappa\big(\hat{\iota}_i(s)\big)\cdot\Big(\nabla G\big(\mathbf x_i,\hat{\iota}_i(s)\big)\cdot\hat{\mathbf{n}}\big(\hat{\iota}_i(s)\big)\Big)\sqrt{1+\big(p_i'(s)\big)^2}ds\\
&+\sum_{j\neq i}\int_{\Delta s_{-j}}^{\Delta s_j} \kappa\big(\hat{\iota}_j(s)\big)\Big(\nabla G\big(\mathbf x_i,\hat{\iota}_j(s)\big)\cdot \hat{\mathbf{n}}\big(\hat{\iota}_j(s)
\big)\Big)\sqrt{1+\big(p_j'(s)\big)^2}ds\Bigg].
\end{aligned}
\end{equation}

Subsequently, we employ standard quadrature rules to derive a discrete approximation of $AV_n(\mathbf x_i)$ and $b(\mathbf x_i)$. In the numerical implementation, Simpson’s rule is used to approximate $b(\mathbf x)$ yielding the vector $\mathbf b$, while the trapezoidal rule is used to approximate $AV_n(\mathbf x)$, resulting in the matrix $\mathbf A$. Notice that the first integrals in \eqref{eq:AVdetail} and \eqref{eqn:b_int} are singular when $s=0$, because $\hat{\iota}_i(0)=\mathbf x_i$. For additional details on the treatment of the singularities, we refer to our previous work \cite{zhang2026geometric}.
Therefore, we obtain the linear system $\mathbf A \mathbf{\tilde{V}}_n = \mathbf b$, which can be solved for the discrete approximate $\tilde{\mathbf V}_n$ of the normal velocity. Here, $\tilde{\mathbf V}_n$ is an approximation to $\mathbf V_n$, whose $i$th component is $(\mathbf{V}_n)_i = V_n(\mathbf x_i)$.

Finally, we apply to time integration schemes to discretize \eqref{velocity}, $\frac{d\mathbf x}{dt}=\tilde{\mathbf V}_n\hat{\mathbf n}$ with $\tilde{\mathbf V}_n=\mathbf A^{-1}\mathbf b$. In the following numerical experiments, we will consider both the Forward Euler scheme and the second-order Runge-Kutta scheme.

\section{Numerical method for non-zero source problems}\label{sec:non-zero-source}
For the Hele-Shaw free boundary problem with $f\neq0$, the BIE \eqref{BIE} contains a volume integral over the evolving domain $\Omega(t)$. A direct quadrature approximation to this term will require the construction of a mesh on $\Omega(t)$ at every time step as the boundary evolves, which is computationally expensive. In this section, we introduce an alternative method to circumvent this difficulty. Specifically, we consider
the following reformulation to the original system \eqref{hele-shaw}. Let us decompose  $p=p_1+p_2$ such that,
\begin{equation}\label{eqn:system1}
    -\Delta p_1=f \quad\text{on}~\overline{\Omega},
\end{equation}

and
\begin{equation}\label{eqn:system2}
    \left\{
    \begin{array}{rcll}
    -\Delta p_2 &=& 0 &\text{on}~\Omega(t),\\
    p_2 & =& \kappa - \left.p_1\right|_{\Gamma}&\text{on}~\Gamma(t),\\
    \frac{\partial p_2}{\partial \mathbf n}&=&-V_n-\left.\frac{\partial p_1}{\partial \mathbf n}\right|_{\Gamma}&\text{on}~\Gamma(t),
    \end{array}
    \right.
\end{equation}
where $\Omega(t)\subset\overline{\Omega}$ for all $t\geq 0$ and $\overline{\Omega}$ is independent of time.

This decomposition allows one to solve \eqref{eqn:system1} for $p_1$ on $\overline{\Omega}$ once, which can be done offline, since this is a standard Poisson problem with a fixed domain. Subsequently, we can apply the boundary integral formulation to \eqref{eqn:system2}, yielding following system,
\begin{subnumcases}{\label{BIM2}}
\int_{\Gamma} G(\mathbf x,\mathbf y)\left(V_n+\left.\frac{\partial p_1(\mathbf y)}{\partial \mathbf n(\mathbf y)}\right|_{\Gamma}\right)dS_{\mathbf y}=-\frac{\kappa(\mathbf x)-\left.p_1(\mathbf x)\right|_{\Gamma}}{2}-\int_{\Gamma} \left(\kappa(\mathbf y)-\left.p_1(\mathbf y)\right|_{\Gamma}\right)\frac{\partial G(\mathbf x,\mathbf y)}{\partial\mathbf n(\mathbf y)} dS_{\mathbf y}  ,\quad \mathbf{x},\mathbf{y} \in \Gamma(t),\label{BIE2}\\
\frac{d\mathbf x}{dt} =  V_n(\mathbf x) \mathbf n(\mathbf x), \quad \mathbf x\in\Gamma(t).\label{velocity2}
\end{subnumcases}
Then, we can apply the numerical scheme discussed in Section \ref{sec:GMLS} to solve \eqref{velocity2} to obtain $V_n$.

\begin{remark}
    We don't specify the boundary conditions for $p_1$ in the system \eqref{eqn:system1}. Indeed, for a given source term $f$, the volume potential is uniquely determined by $f$ and is independent of the particular boundary conditions used to define $p_1$. Subtracting \eqref{BIE2} from \eqref{BIE} yields
    \[\int_{\Omega} G(\mathbf x,\mathbf y)f(\mathbf y)dV_{\mathbf y}=\frac{\left. p_1(\mathbf x)\right|_{\Gamma}}{2}+\int_{\Gamma}\left. p_1(\mathbf y)\right|_{\Gamma}\frac{\partial G(\mathbf x,\mathbf y)}{\partial\mathbf n(\mathbf y)}dS_{\mathbf y}-\int_{\Gamma}G(\mathbf  x,\mathbf y)\left.\frac{\partial p_1(\mathbf y)}{\partial \mathbf n(\mathbf y)}\right|_{\Gamma}dS_{\mathbf y},\]
    which indicates that any change in $p_1$ is determined by the volume potential, which depends solely on the source term $f$. 
\end{remark}

Next, we present a numerical example for which \eqref{eqn:system1} admits analytical solutions for the corresponding source term $f$.
More general cases, where \eqref{eqn:system1} does not admit analytical solutions, are discussed in Section \ref{sec:eigenfunction-approx}.

\paragraph{Free boundary problem with radial source term.}
We consider the free boundary equation with a radial source term,
\begin{equation}\label{eq:FBP_radial_source}
\left\{
\begin{array}{rcll}
-\Delta p &=& r^2& \text{in }\, \Omega(t), \\
p &=& \kappa & \text{on }\, \Gamma(t), \\
\frac{\partial p}{\partial \mathbf n} &=& -V_n& \text{on }\, \Gamma(t),
\end{array}
\right.
\end{equation}
where $r=\sqrt{x_1^2+x_2^2}$ is the radial coordinate. By decomposition, $p=p_1+p_2$, we reformulate \eqref{eq:FBP_radial_source} as 
\begin{equation}\label{eqn:system1_radial_source}
    -\Delta p_1=r^2 \quad\text{on}~\overline{\Omega},
\end{equation}
where $p_1(r) = -\frac{1}{16}r^4$ is a particular solution of the system \eqref{eqn:system1_radial_source},
and
\begin{equation}\label{eqn:system2_radial_source}
    \left\{
    \begin{array}{rcll}
    -\Delta p_2 &=& 0 &\text{on}~\Omega(t),\\
    p_2 & =& \kappa - \left.p_1\right|_{\Gamma}&\text{on}~\Gamma(t),\\
    \frac{\partial p_2}{\partial \mathbf n}&=&-V_n-\left.\frac{\partial p_1}{\partial \mathbf n}\right|_{\Gamma}&\text{on}~\Gamma(t).
    \end{array}
    \right.
\end{equation}
We consider the initial condition in which $\Omega(t)$ is a disk of radius 2 at $t=0$ and $\overline{\Omega}$ is a disk of radius $\overline{R}=3$, both centered at the origin. Then the corresponding initial point clouds are defined as
\begin{equation}\label{eq:IC1}
\mathbf x(0)=\Big(x_1(0),x_2(0)\Big)=\Big(2\cos(\theta_j),2\sin(\theta_j)\Big),\quad \theta_j\in[0,2\pi],
\end{equation}
where $\theta_j = \frac{2\pi j}{N}$ and $N$ is the number of sampling points. For the circular case, the general solution of $-\Delta p_2=0$ is $p_2(r)=C_1\log(r)+C_2$. To ensure $p_2(r)$ is finite at $r=0$, we have $C_1=0$ and $p_2(r) = C_2=\kappa-\left.p_1\right|_{\Gamma}$ which implies $V_n+\left.\frac{\partial p_1}{\partial \mathbf n}\right|_{\Gamma}=0$. Therefore, $\frac{dR(t)}{dt}=V_n=-\left.\frac{\partial p_1}{\partial\mathbf n}\right|_{\Gamma}=\frac{R(t)^3}{4}$, where $R(t)$ denotes the radius of the evolving domain $\Omega(t)$.

Therefore, the solution to \eqref{eqn:system2_radial_source} at any time $t$ is given by
\begin{equation*}
\begin{aligned}
    R(t) &= \left(\frac{1}{4}-\frac{1}{2}t\right)^{-\frac{1}{2}},\quad t<0.5,\\
    \mathbf x(t)&=\Big(x_1(t),x_2(t)\Big)=\Big(R(t)\cos(\theta),R(t)\sin(\theta)\Big),\quad\theta\in[0,2\pi].
    \end{aligned}
\end{equation*}
Then we conduct the mesh refinement tests to check the order of spatial convergence and temporal convergence. First, we apply our method to discretize \eqref{eqn:system2_radial_source} directly on point clouds with $N=100$, $200$, $400$, $800$, and use the Forward Euler with time step $\Delta t=10^{-5}$. We quantify the numerical error in approximating $V_n(\mathbf x)$ with, \[\|e_V\|_{L^2}=\sqrt{\frac{1}{|\Gamma|}\int_\Gamma \|\tilde{\mathbf V}_n(\mathbf x)- V_n(\mathbf x)\|^2dS_{\mathbf x}}\approx\sqrt{\frac{1}{N}\sum_{i=1}^N \Big((\tilde{\mathbf V}_n)_i- V_n(\mathbf x_i)\Big)^2},\] 
at the time $t=5\times10^{-2}$. In Figure \ref{fig:space-convergence}, we show the error rate is dominated by the error induced by the quadrature rule in the estimation of $b$, consistent with the results obtained in our previous work in \cite{zhang2026geometric} for no source term, $f=0$. Next, we conduct the time convergence studies with $N=400$ and calculate the numerical solutions to $t=5\times10^{-2}$ with various time steps. We define the error for radius of the disk with $\|e_R\|_{\ell^2}=\|\tilde{R}-R(t)\|_2$,
where $\tilde{R}$ denotes approximate radius in time. Figure \ref{fig:time-convergence} indicates the first-order accuracy in time with the Forward Euler method. Also, the evolution dynamics with $N=400$, $\Delta t=10^{-5}$, and Forward Euler method are shown in Figure \ref{fig:boundary-dynamics} which depicts the boundary at various times, $t=0$, $0.05$, and $0.1$. The absolute error in the radius over time {using the forward Euler and second-order Runge-Kutta time integration schemes} is shown in Figure \ref{fig:radius-error}.

\begin{figure}[htbp]
\center
\subfigure[Spatial mesh refinement test for $V_n$.]{\includegraphics[width=0.45\linewidth]{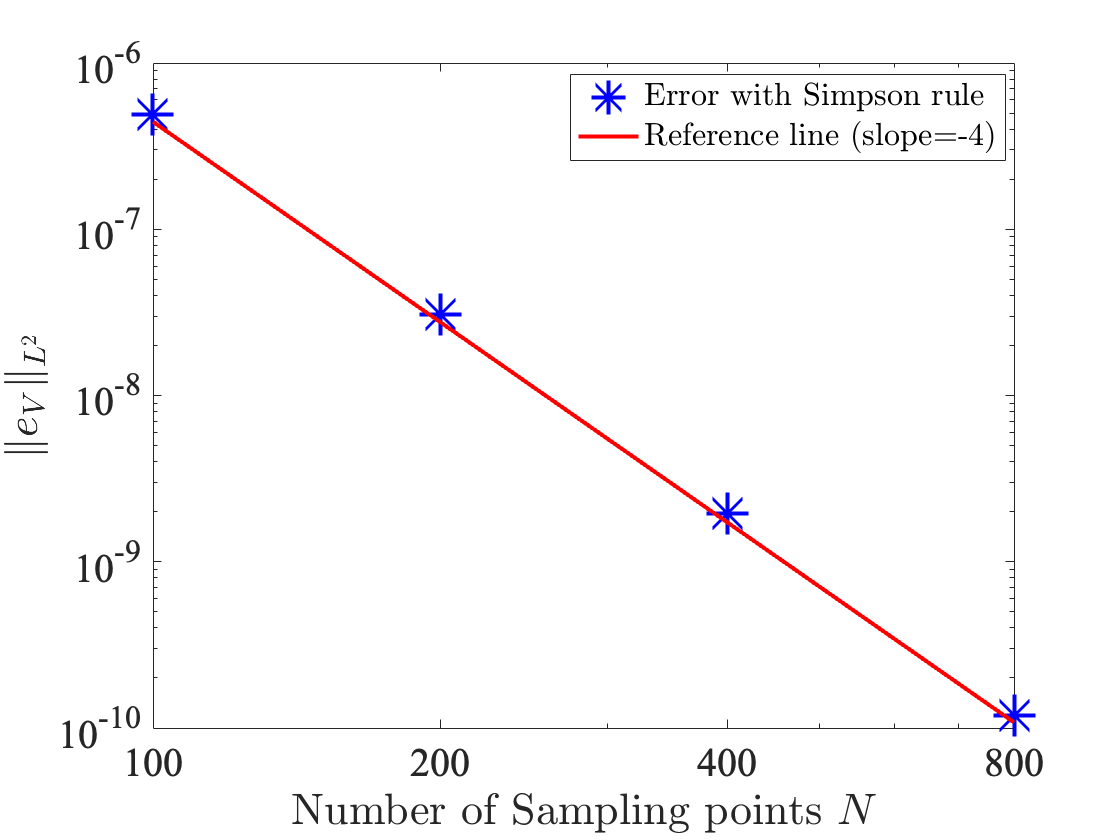}
\label{fig:space-convergence}}
\subfigure[Temporal mesh refinement test for radius of the circle.]{\includegraphics[width=0.45\textwidth]{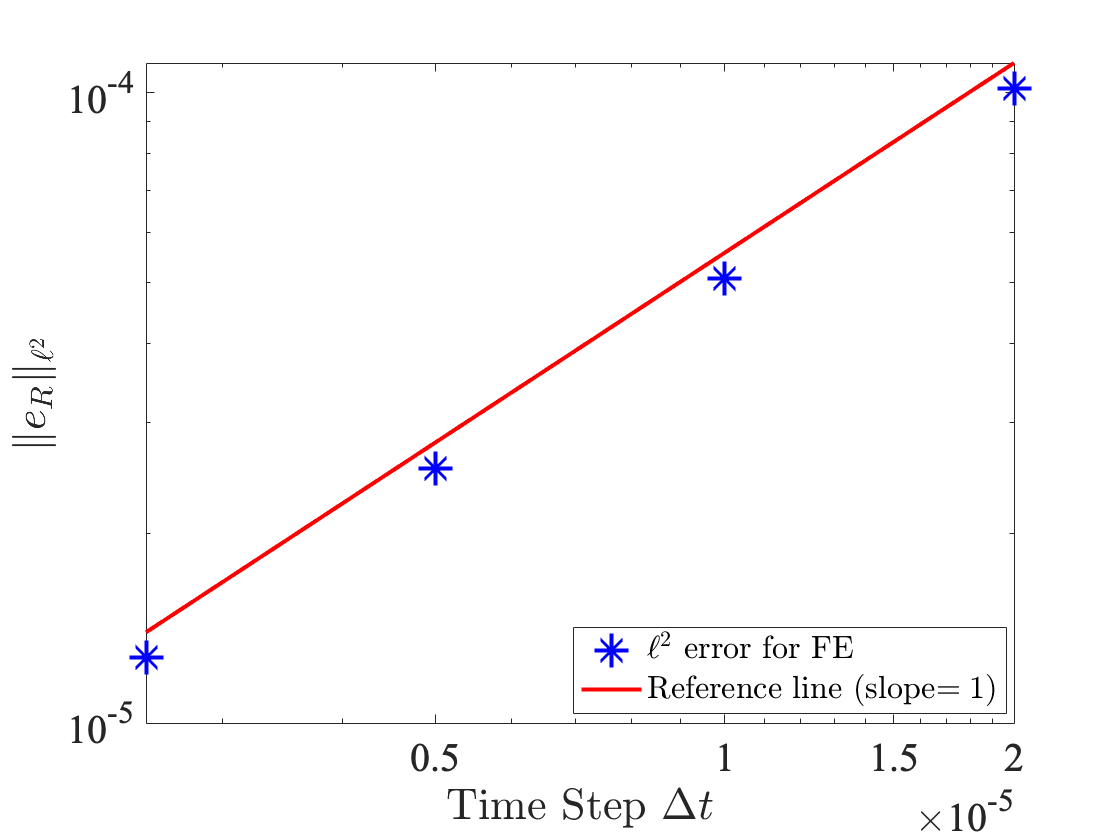}
\label{fig:time-convergence}}
\caption{Mesh refinement tests.}
\label{fig:convergence-test}
\end{figure}

\begin{figure}[htbp]
\center
\subfigure[Evolution of the circle boundary over time.]{\includegraphics[width=0.45\textwidth]{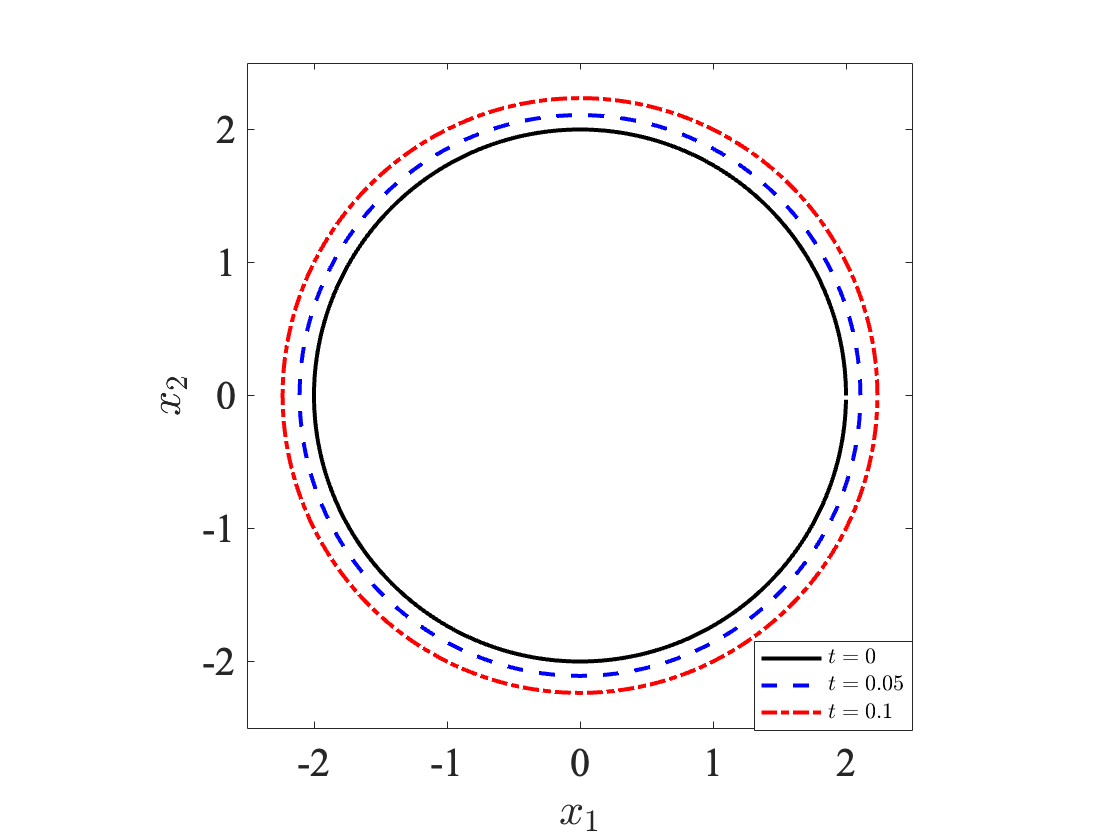}\label{fig:boundary-dynamics}}
\subfigure[Error for the circle radius over time.]{\includegraphics[width=0.45\textwidth]{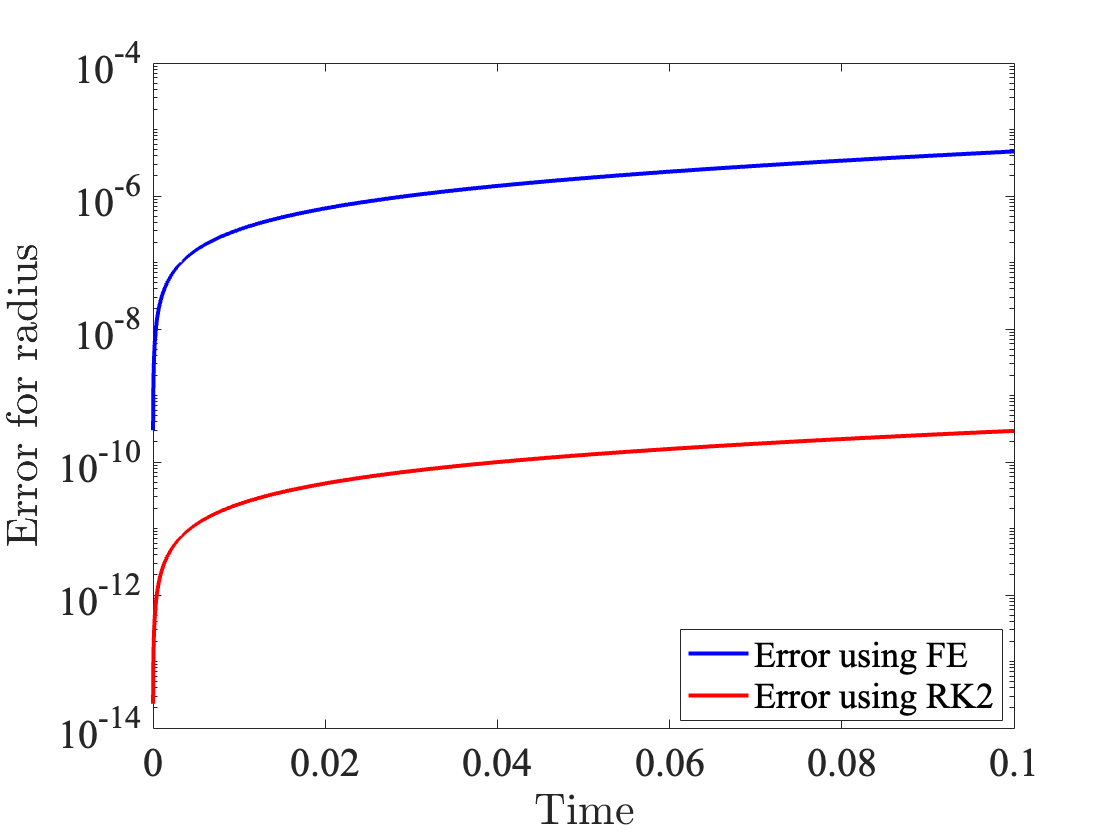}\label{fig:radius-error}}
\caption{The evolution dynamics for the circular case. In (a), the profiles of $\mathbf x(t)$ at $t=0$, $0.05$, $0.1$ are shown; (b) The errors for radius $R(t)$ are plotted at a time interval $[0,0.1]$.}
\label{fig:circle}
\end{figure}

\section{Eigenfunction approximation of the source term}\label{sec:eigenfunction-approx}
In this section, we consider the case in which the system \eqref{eqn:system1} does not admit an analytical solution. To approximate the radial source function $f(r)$ on the disk $\overline{\Omega}$ of radius $\overline{R}$ with $r\in[0,\overline{R}]$, we employ the eigenfunction expansion method. The resulting eigenfunction expansion yields an approximation of $f$, while the corresponding approximation of the solution $p_1$ is obtained through the associated eigenvalue relation. We then establish approximation error estimates for the proposed eigenfunction-based approximation.

\subsection{Eigenfunction expansion}\label{subsec:eigenfunction}
To construct the eigenfunction-based approximation, we consider the following homogeneous Neumann eigenvalue problem associated with the negative Laplacian on the disk $\overline{\Omega}$ of radius $\overline{R}$,
\[-\Delta\phi_m=\lambda_m\phi_m,\quad \frac{\partial\phi_m}{\partial\mathbf n}=0~\text{on}~\partial\overline{\Omega},\]
where $\Delta$ denotes the two-dimensional radial Laplacian. For radial functions, the eigenfunctions are given by 
\BEA\phi_m(r)=J_0\left(\frac{\beta_m r}{\overline{R}}\right),\quad \lambda_m=\left(\frac{\beta_m}{\overline{R}}\right)^2,\quad m=1,2,\cdots,\label{eigensolutions}\EEA
where $J_0$ is the Bessel function of the first kind of order zero and $\beta_m$ denotes the $m$-th positive root of $J_1$. The eigenfunctions $\{\phi_m\}_{m=1}^\infty$ form an orthogonal basis of the mean-zero subspace of $L^2_r(0,\overline{R})$, which is a weighted space with inner product defined as $\langle\phi_i,\phi_j\rangle_{L^2_r}=\int_0^{\overline{R}}\phi_i(r)\phi_j(r)rdr$. 

Since the Neumann eigenfunctions span only the mean-zero subspace, the source function must satisfy the compatibility condition,
\BEA\int_0^{\overline{R}}f(r)rdr=0.\label{eq:meanzero}
\EEA
To accommodate a general source function $f$, we decompose 
\BEA
p_1=\widehat{p}_1+\overline{p}_1,\quad f=f_0+\overline{f},\label{eq:meanzerodecomposition}
\EEA
where $\overline{f}=\frac{2}{\overline{R}^2}\int_0^{\overline{R}}f(r)rdr$. By construction, $\int_0^{\overline{R}}f_0(r)rdr=0$.

Then the Poisson problem in \eqref{eqn:system1} with radial function $f$ becomes
\[
-\Delta \widehat{p}_1=f_0,\quad-\Delta\overline{p}_1=\overline{f}.\]

Since $\overline{f}$ is a constant, $\overline{p}_1(r)=-\frac{\overline{f}}{4}r^2$ is a particular solution. Therefore, it remains to approximate $\widehat{p}_1$. By construction, $f_0$ satisfies the compatibility condition and hence it can be approximated by the truncated eigenfunction expansion
\BEA
f_0(r)\approx f_{0,M}(r)=\sum_{m=1}^M b_m\phi_m(r),\text{ where }b_m=\frac{\int_0^{\overline{R}}f_0(r)\phi_m(r)rdr}{\int_0^{\overline{R}} \phi_m^2(r)rdr}\label{eq:f_0M}\EEA
In the numerical implementation, the integrals are approximated using Gauss-Legendre quadrature. 
Together, we effectively approximate the source term as,
$$f\approx f_M:=f_{0,M}+\overline{f}.$$
Then the approximation of the solution $\widehat{p}_1$ follows directly from the eigenvalue relation 
\BEA\widehat{p}_1\approx\widehat{p}_{1,M}=\sum_{m=1}^Ma_m\phi_m,\quad a_m=\frac{b_m}{\lambda_m}.\label{eq:hatp1M}\EEA
Finally, we have
\[p_1(r)\approx p_{1,M}(r)=\widehat{p}_{1,M}(r)+\overline{p}_1(r) = \sum_{m=1}^M\frac{b_m}{\lambda_m}\phi_m(r) -\frac{\overline{f}}{4}r^2.\]

\paragraph{Free boundary problem with radial source term $f(r)=r^2$.} 
We next examine the effectiveness of the eigenfunction approximation. To facilitate comparison with the previous example, we consider the same system \eqref{eq:FBP_radial_source} with source term $f(r)=r^2$. Instead of using the exact source term, we approximate $f_0=f-\overline{f}$ by the truncated eigenfunction expansion, where $\overline{f}=\frac{2}{\overline{R}^2}\int_0^{\overline{R}} f(r)rdr=\frac{\overline{R}^2}{2}$ is the weighted radial mean. Similar to previous experiment, we set $\overline{R}=3$. We apply the proposed eigenfunction approximation using Gauss-Legendre quadrature with $N_r=2000$ quadrature points and consider $M=50, 100, 200, 400$ eigenfunctions. To assess the approximation accuracy, we evaluate the numerical error in the source function $f$ and the corresponding solution $p_1$ using the weighted $L^2$-norm,
\[\|e_f\|_{L^2_r}=\left(\int_0^{\overline{R}} \left(f(r)-f_M(r)\right)^2rdr\right)^{\frac{1}{2}},\quad \|e_p\|_{L^2_r}=\left(\int_0^{\overline{R}} \left(p_1(r)-p_{1,M}(r)\right)^2rdr\right)^{\frac{1}{2}}.\]  
We note that the Neumann problem determines $\widehat{p}_1$ only up to an additive constant. Therefore, we impose the normalization condition $\int_0^{\overline{R}} \widehat{p}_1(r)rdr=0$. The same normalization is applied to the approximation $\widehat{p}_{1,M}$ to have zero weighted mean before computing the approximation error. First, we investigate the convergence behavior of $\|e_f\|_{L^2_r}$ and $\|e_p\|_{L^2_r}$ as the number of eigenfunctions functions $M$ increases. As shown in Figure \ref{fig:r2_approx_error}, the source approximation error $\|e_f\|_{L^2_r}$ decays approximately as $M^{-1.5}$ while the solution error $\|e_p\|_{L^2_r}$ decays approximately as $M^{-3.5}$. The factor of $M^2$ faster convergence for $\|e_p\|_{L^2_r}$ will be clarified in the next section.

\begin{figure}[htbp]
\center
\subfigure[]{\includegraphics[width=0.45\linewidth]{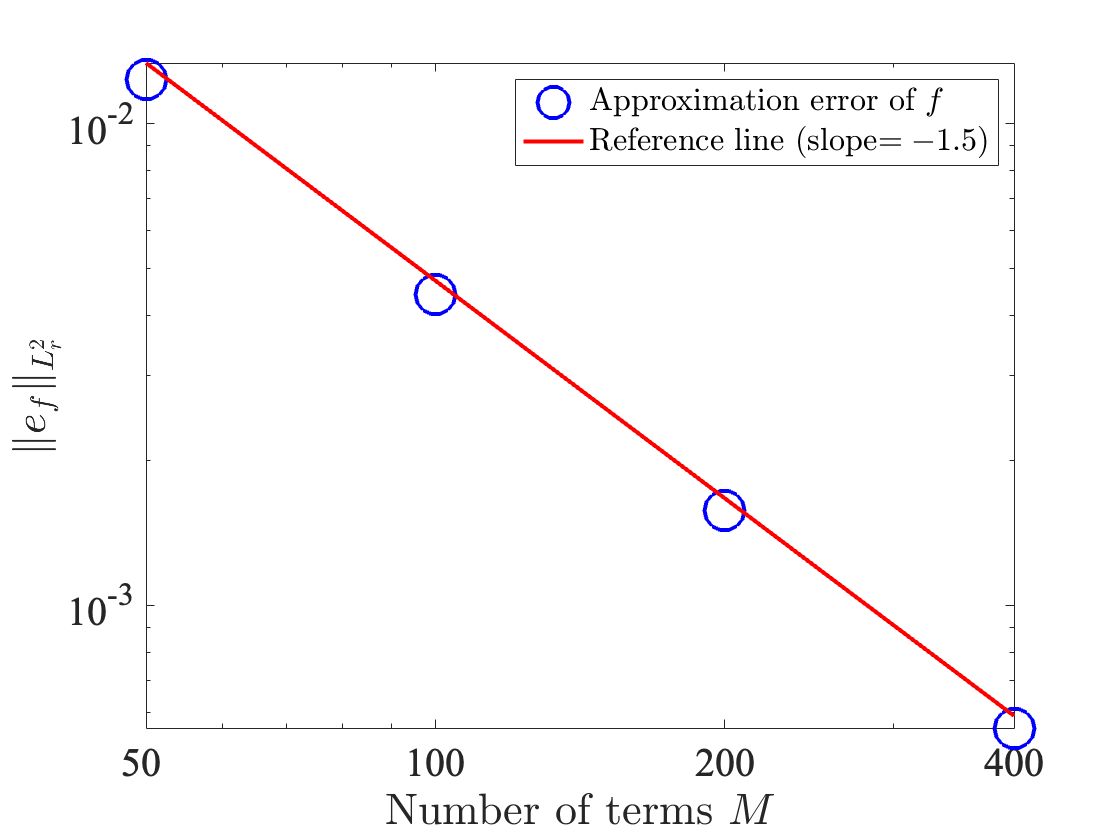}
\label{fig:error_f}}
\subfigure[]{\includegraphics[width=0.45\textwidth]{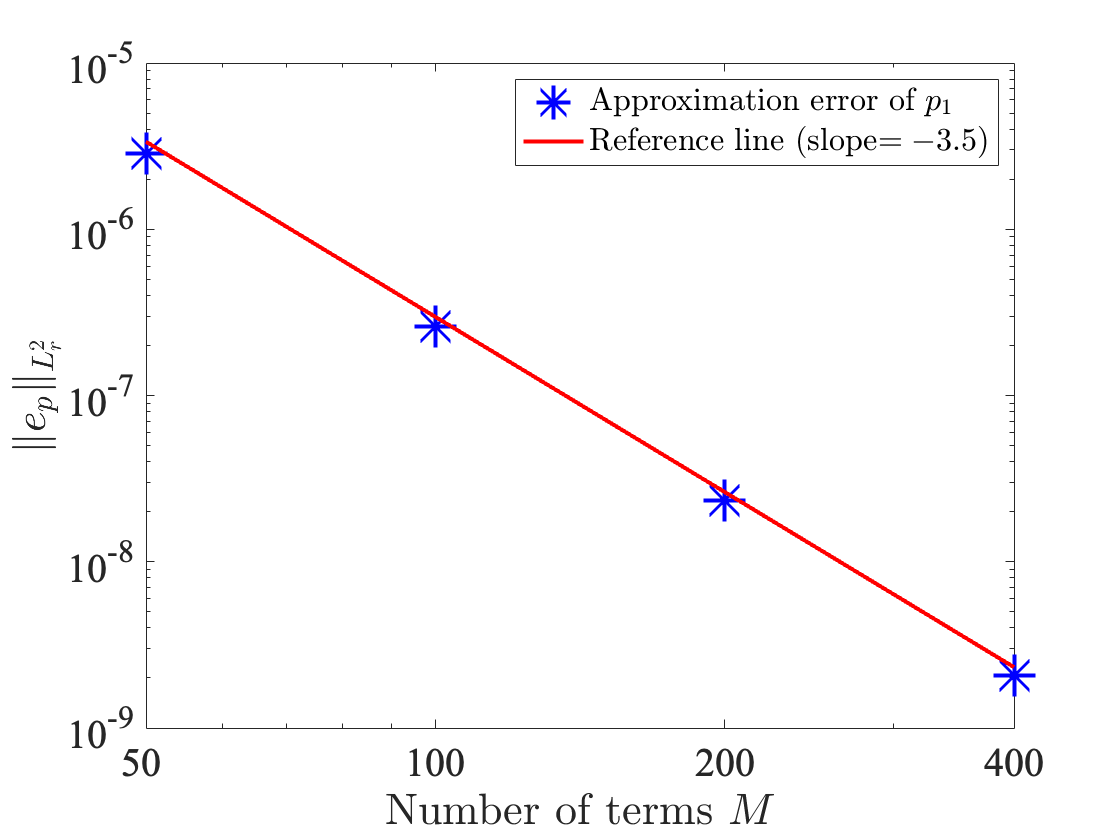}
\label{fig:error_p1}}
\caption{Weighted $L^2$ errors for the approximation of the source function $r^2$ and the corresponding solution $p_1$.}
\label{fig:r2_approx_error}
\end{figure}

Next, we investigate the dominant error contribution by comparing the total numerical error produced by the proposed eigenfunction approximation with that obtained when the exact source term is used in the computation. The total error in solving \eqref{eq:FBP_radial_source} by approximating $f$ is measured by the difference between the numerical radius and the exact radius. We first fix the number of eigenfunctions at $M=200$. Figure \ref{fig:error_fd} shows that the time discretization error dominates when Forward Euler scheme is employed. In contrast, Figure \ref{fig:error_rk2} indicates that the error associated with the eigenfunction approximation becomes dominant when the second-order Runge-Kutta method is used. The resulting radius error exhibits an oscillatory behavior in time, reflecting the truncation error introduced by the Bessel eigenfunction approximation. When the number of eigenfunctions is increased to $M=1500$, the eigenfunction approximation error continues to dominate the total error, while the oscillatory behavior is substantially reduced.

\begin{figure}[htbp]
\center
\subfigure[Radius errors using the Forward Euler scheme.]{\includegraphics[width=0.45\linewidth]{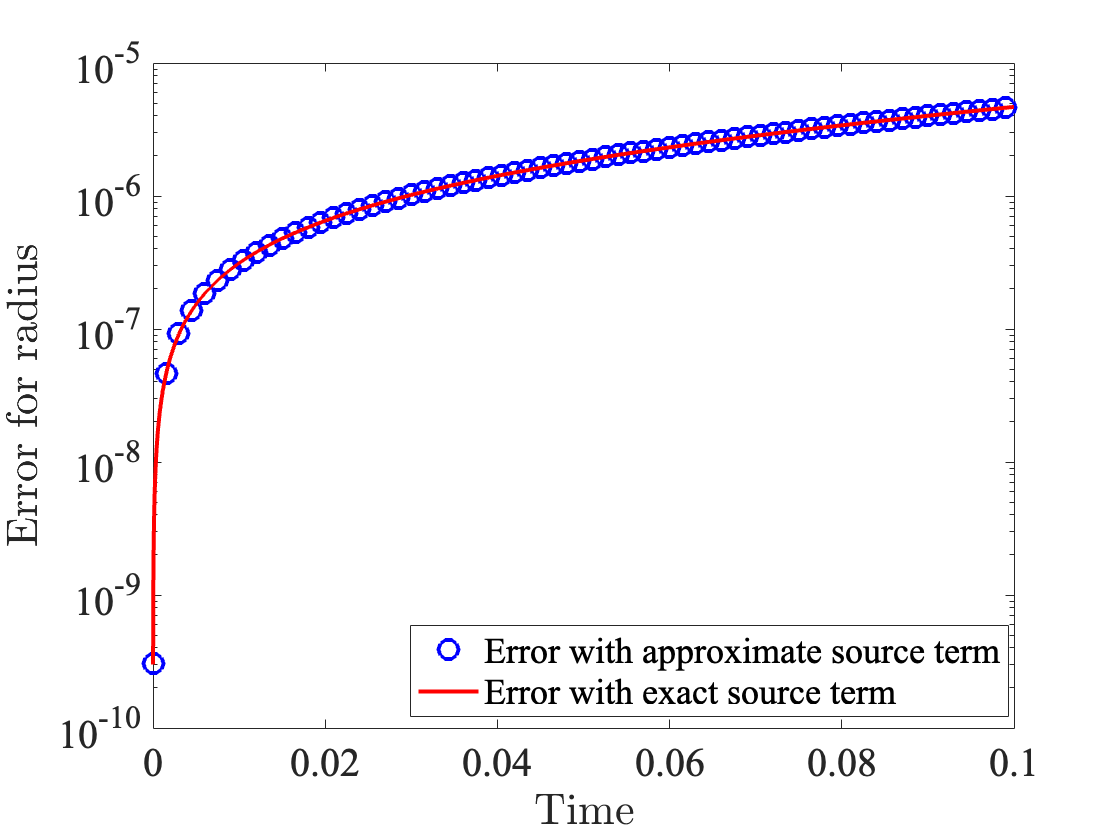}
\label{fig:error_fd}}
\subfigure[Radius error using the second-order Runge-Kutta scheme.]{\includegraphics[width=0.45\textwidth]{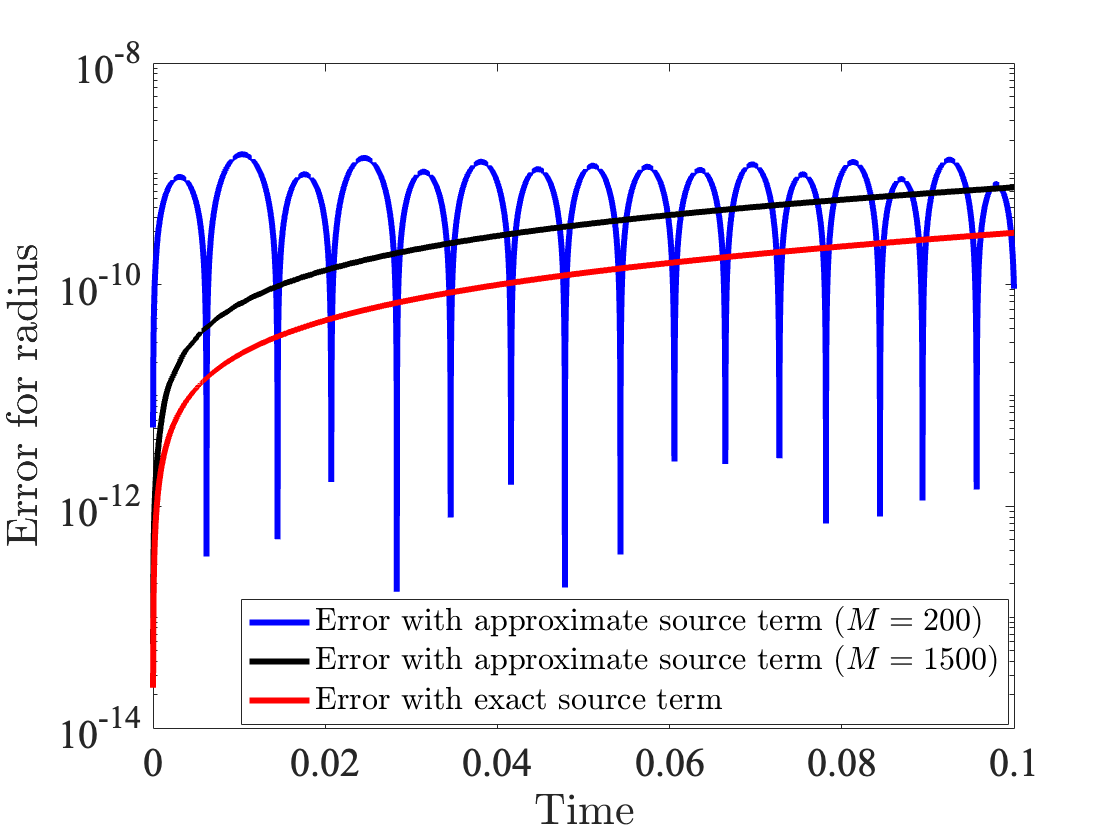}
\label{fig:error_rk2}}
\caption{Comparison of the radius errors obtained using the exact source term and the eigenfunction approximation. The Forward Euler scheme is shown in (a), while the second-order Runge-Kutta scheme is shown in (b).}
\label{fig:total}
\end{figure}
{
Although $r^2$ is analytic, it does not satisfy the Neumann boundary conditions. To investigate the influence of these compatibility conditions on the convergence rate, we next consider the source term which satisfies the homogeneous Neumann boundary condition at $r=\overline{R}$.

\paragraph{Free boundary problem with radial source term $f(r)=\cos\left(\frac{2\pi r^2}{\overline{R}^2}\right)$.} Since $f(r)=\cos\left(\frac{2\pi r^2}{\overline{R}^2}\right)$ is already mean-zero, we approximate $f$ by the truncated eigenfunction expansion. We use the same numerical settings as in the previous example, namely $\overline{R}=3$, $N_r=2000$ Gauss-Legendre quadrature points, and $M=50$, $100$, $200$, and $400$ eigenfunctions. The corresponding solution is $\widehat{p}_1=-\frac{\overline{R}^2}{8\pi}Si\left(\frac{2\pi r^2}{\overline{R}^2}\right) +C$ where $Si(x)=\int_0^x \frac{\sin z}{z}dz$ is the sine integral and the constant $C$ is determined by the normalization condition $\int_0^{\overline{R}} \widehat{p}_1(r)rdr=0$. We then investigate the convergence behaviour of the source approximation error $\|e_f\|_{L_r^2}$ and the corresponding solution error $\|e_p\|_{L_r^2}$, where the latter is computed with respect to the reference solution evaluated in MATLAB. As shown in Figure \ref{fig:Neumann_approx_error}, the source approximation error decays at an approximate rate of $M^{-3.5}$, which is two orders higher than that observed for $f(r)=r^2$. The corresponding solution error exhibits an approximate convergence rate of $M^{-5.5}$. These numerical results indicate that the convergence rate depends on whether the source term satisfies the compatibility conditions associated with the Neumann Laplacian. A theoretical explanation is provided in Section \ref{sec:error-estimate}.

\begin{figure}[htbp]
\center
\subfigure[]{\includegraphics[width=0.45\linewidth]{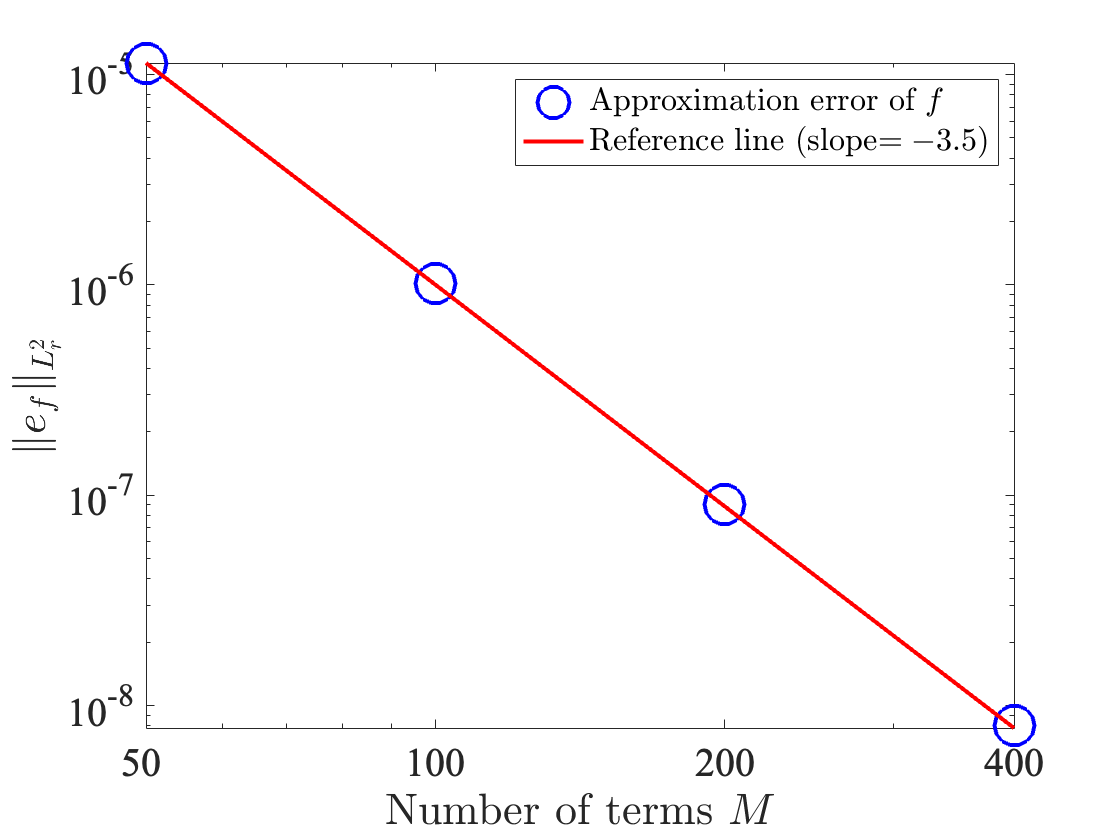}
\label{fig:Neumann_error_f}}
\subfigure[]{\includegraphics[width=0.45\textwidth]{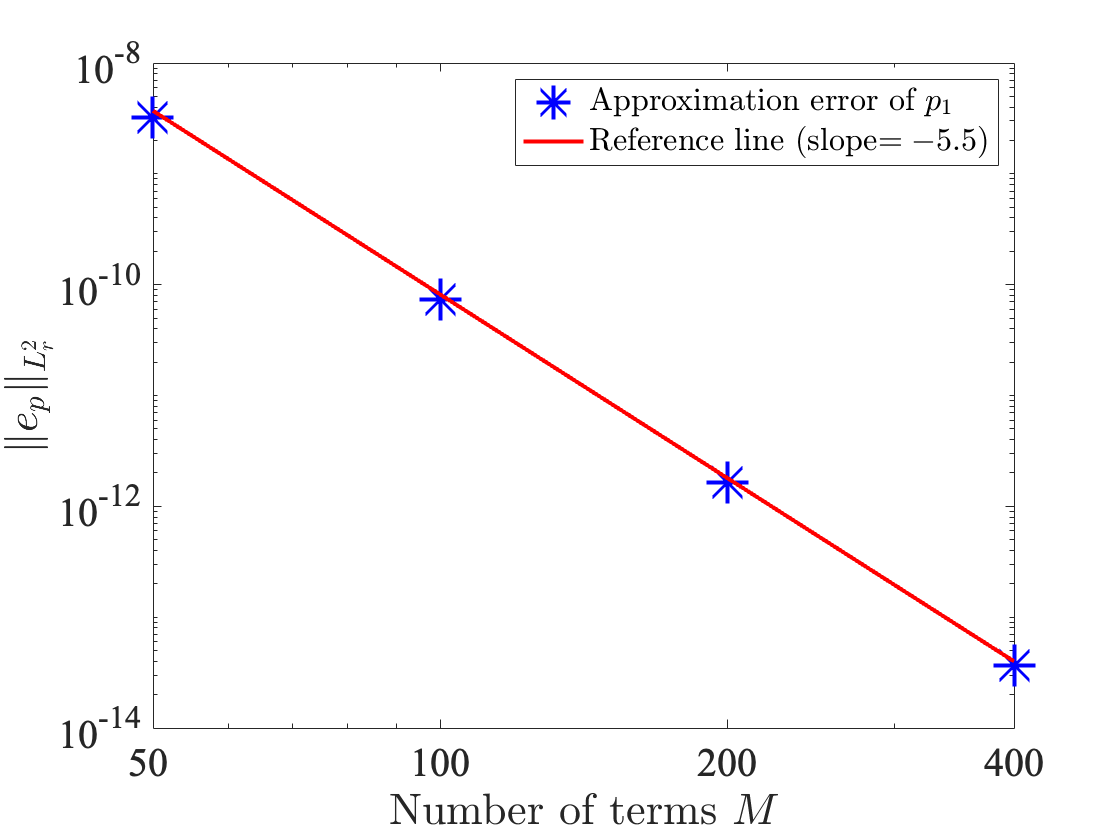}
\label{fig:Neumann_error_p1}}
\caption{Weighted $L^2$ errors for the approximation of the source function $\cos\left(\frac{2\pi r^2}{\overline{R^2}}\right)$ and the corresponding solution $p_1$.}
\label{fig:Neumann_approx_error}
\end{figure}
}

\subsection{Error analysis}\label{sec:error-estimate}
The error analysis of the geometric local parameterization  method for solving BIE with $f=0$ has already been established in our previous work \cite{zhang2026geometric}.
Here, we provide an error analysis for eigenfunction approximation. 

{
\begin{thm}\label{thm:basis-error}
Let $s\geq 2$  be an even integer. Assume that $f\in D_{s/2-1}$, where, for $s=2$,
\[D_0:=\left\{u\in H_r^2(0,\overline{R}): (-\Delta)^i u\in L_r^2(0,\overline{R}),\quad i=0,1\right\},\]
and, for every even integer $s\geq 4$,
\[D_{s/2-1}:=\left\{u \in H_r^s(0,\overline{R}):(-\Delta)^i u\in L_r^2(0,\overline{R}),\quad i=0,1,\cdots,\frac{s}{2},\quad\partial_r\left((-\Delta)^j u\right)(\overline{R})=0,\quad j=0,1,\cdots,\frac{s}{2}-2\right\}.\]
Let $f = f_0 +\overline{f}$, where $f_0$ is the mean zero component and $\overline{f}$ is a constant defined as in \eqref{eq:meanzerodecomposition}. Let $\widehat{p}_1$ be the corresponding mean-zero radial solution of
\[-\Delta \widehat{p}_1=f_0~\text{in}~(0,\overline{R}),\quad \partial_r \widehat{p}_{1}(\overline{R})=0,\quad \int_0^{R_1}\widehat{p}_1(r)rdr=0,\]
and let $\widehat{p}_{1,M}$ and $f_{0,M}$ be defined as in \eqref{eq:hatp1M} and \eqref{eq:f_0M}, respectively.
where $\{\lambda_m,\phi_m\}_{m\geq1}$ are the Neumann eigenfunctions of $-\Delta$ on the disk of radius $\overline{R}$ as defined in \eqref{eigensolutions}.
Then
\[\|f_0 -f_{0,M}\|_{L^2_r(0,\overline{R})}\leq C_1\overline{R}^s M^{-s+\frac{1}{2}},\]
and
\[\|\widehat{p}_1-\widehat{p}_{1,M}\|_{L^2_r(0,\overline{R})}\leq C_2\overline{R}^{s+2}M^{-(s+\frac{3}{2})},\]
as $M \to \infty$, where $C_1=\|(-\Delta)^{s/2}f_0\|_{L_r^2}+\left|\partial_r\left((-\Delta)^{s/2-1}f_0\right)(\overline{R})\right|$ and $C_2=\pi^2C_1$. 
\end{thm}

\begin{proof}
With $b_m\|\phi_m\|_{L^2_r}^2 = \left\langle f_0,\phi_m\right\rangle_{L^2_r}$, $(-\Delta)^{s/2}\phi_m=\lambda_m^{s/2}\phi_m$, and $\phi_m'(\overline{R})=J_1(\beta_m)=0$, repeated integration by parts yields 
\[b_m\|\phi_m\|_{L^2}^2=\lambda_m^{-s/2}\langle f_0,(-\Delta)^{s/2}\phi_m\rangle_{L^2_r}=\lambda_m^{-s/2}\left(\langle(-\Delta)^{s/2} f_0,\phi_m\rangle_{L^2_r} +\overline{R}\partial_r\left((-\Delta)^{s/2-1}f_0\right)(\overline{R})\phi_m(\overline{R})\right).
\]
By the triangle inequality and the Cauchy-Schwarz inequality, we obtain
\begin{equation*}
\begin{aligned}
|b_m|&\leq\frac{\lambda_m^{-s/2}\left(\left|\langle(-\Delta)^{s/2} f_0,\phi_m\rangle_{L^2_r} \right|+\overline{R}\left|\partial_r\left((-\Delta)^{s/2-1}f_0\right)(\overline{R})\right|\left|\phi_m(\overline{R})\right|\right)}{\|\phi_m\|^2_{L_r^2}}\\
   &\leq \frac{\lambda_m^{-s/2}\left(\|(-\Delta)^{s/2}f_0\|_{L_r^2}\|\phi_m\|_{L^2_r}+\overline{R}\left|\partial_r\left((-\Delta)^{s/2-1}f_0\right)(\overline{R})\right|\left|\phi_m(\overline{R})\right|\right)}{\|\phi_m\|^2_{L_r^2}}.
\end{aligned}
\end{equation*}
From the classical result, the McMahon's asymptotic expansions (see \S10.21(vi) of \cite{olver2010nist}) for large zeros in \eqref{eigensolutions}, we have
    \[
    \beta_m = m\pi + \frac{\pi}{4} + O(m^{-1}),
    \]
    as $m \to\infty$, which implies that,
    \BEA
    \lambda_{m} \asymp \frac{m^2}{\overline{R}^2}, \label{lambdaasymptotic}
    \EEA
    for large $m$. By Lommel's integrals (see \S94 of \cite{bowman2012introduction}) and asymptotic expansions (see \S80 of \cite{bowman2012introduction}), we have $\|\phi_m\|_{L^2_r}^2=\frac{\overline{R}^2}{2}J_0^2(\beta_m)$ and $J_0(\beta_m)\asymp\sqrt{\frac{2}{\pi\beta_m}}\cos(\beta_m-\frac{\pi}{4})$, which implies
    \[\|\phi_m\|_{L^2_r}\asymp \overline{R}m^{-1/2},\quad \left|\phi_m(\overline{R})\right|=\left|J_0(\beta_m)\right|\asymp m^{-1/2}.\]
    By orthogonality, we have
\BEA
\|f_0-f_{0,M}\|_{L^2_r}=\left(\sum_{m>M} |b_m|^2\|\phi_m\|_{L^2_r}^2\right)^{\frac{1}{2}}= \left(\sum_{m>M}\lambda_m^{-s}\lambda_m^{s}|b_m|^2\|\phi_m\|_{L^2_r}^2\right)^{\frac{1}{2}}\leq C_1\left(\sum_{m>M}\lambda_m^{-s}\right)^{\frac{1}{2}}\leq C_1\overline{R}^s\left(\sum_{m>M} m^{-2s}\right)^{\frac{1}{2}},\label{eq:id1}
\EEA
where $C_1=\|(-\Delta)^{s/2}f_0\|_{L_r^2}+\left|\partial_r\left((-\Delta)^{s/2-1}f_0\right)(\overline{R})\right|$.

Using the fact that
$\sum_{m>M}m^{-2s}=\int_M^{\infty}x^{-2s}dx+O(M^{-2s})=\frac{1}{2s-1}M^{-2s+1}+O(M^{-2s})$, we obtain
\BEA
\|f_0-f_{0,M}\|_{L^2_r}\leq C_1\overline{R}^s M^{-s+\frac{1}{2}}.\label{eq:error_f0}
\EEA
 Now consider $-\Delta \widehat{p}_1=f_0$, we have by orthogonality and the equality in \eqref{eq:id1},
     \[\|\widehat{p}_1-\widehat{p}_{1,M}\|_{L^2_r}^2=\sum_{m>M}\frac{|b_m|^2}{\lambda_m^2}\|\phi_m\|_{L^2_r}^2\leq\lambda_{M+1}^{-2}\sum_{m>M}|b_m|^2 \|\phi_m\|_{L^2_r}^2= \lambda_{M+1}^{-2} \|f_0 -f_{0,M}\|_{L^2_r}^2.\]
    Using \eqref{lambdaasymptotic}, \eqref{eq:error_f0},  the proof is complete.
\end{proof}
\begin{remark}
    The convergence rates observed in the numerical examples are consistent with the error estimates established in Theorem \ref{thm:basis-error}. Specifically, for the source term $f(r)=r^2$, we have $f\in D_0$, corresponding to $s=2$. The observed convergence rate agrees with the estimate in Theorem \ref{thm:basis-error}, $\|e_f\|_{L_r^2}=O(M^{-s+1/2})=O(M^{-1.5})$. Similarly, for $f(r)=\cos\left(\frac{2\pi r^2}{\overline{R}^2}\right)$, we have $f\in D_1$, corresponding to $s=4$. Theorem \ref{thm:basis-error} predicts $\|e_f\|_{L_r^2}=O(M^{-3.5})$, which is in agreement with the observed convergence rate.
\end{remark}
}

Numerically, since we approximate $b_m$ with Gauss-Legendre with $N_r$ quadrature, we have the following quadrature error for each coefficient.

\begin{lem}\label{lem:error_bm}
Define $g_0(r)= r f_0(r)$ and $\varphi_m(r) = \sqrt{r}\phi_m(r)$, where {$f_0 \in D_{s/2-1}$} is defined as in Theorem~\ref{thm:basis-error}.
Let $b_m$, defined in \eqref{eq:f_0M}, be approximated using an order-$q$ Gauss-Legendre quadrature rule $Q_{N_r}$ with $N_r$ quadrature points,

\[
\widetilde{b}_m = \frac{Q_{N_r}(g_0\phi_m)}{Q_{N_r}(\varphi_m^2)},
\]
where $q< s-\frac{1}{2}.$ Then,
{
\[
|b_m - \widetilde{b}_m| \leq C_{s,q} \overline{R}^{-(s+3/2)} m^{(s+1)} N_r^{-q} \|f_0\|_{H^s_r},\]
where $C_{s,q}>0$ is a constant depends on $q$ and $s$, but independent of $\overline{R}, N_r$, and $m$.
}

\end{lem}

\begin{proof}
We first review the asymptotic rates for the Gauss quadrature rule. Specifically, for any $f\in C^q[0,\overline{R}]$ function, we have
\BEA
\left| I(f) - Q_{N_r}(f)\right| \leq C_q N_r^{-q} \|f\|_{C^q},\label{eq:Gauss_algebraic_rate}
\EEA
where $I(f) = \int_0^{\overline{R}} f(r)\,dr$ and $Q_{N_r}(f)$ denotes Gauss-Legendre approximation of the integral with $N_r$ quadrature points. If $f\in C^\infty$, then we have exponential convergence,
\[
\left| I(f) - Q_{N_r}(f)\right| \leq C\rho^{-2N_r},
\]
where $\rho>1$.

From the definitions in the lemma, 
\[
|b_m - \widetilde{b}_m| = \left|\frac{I(g_0\phi_m)}{I(\varphi^2_m)} - \frac{Q_{N_r}(g_0\phi_m)}{Q_{N_r}(\varphi_m^2)} \right| \leq \frac{1}{|Q_{N_r}(\varphi_m^2)|} 
\left(\left|I(g_0\phi_m) -Q_{N_r}(g_0\phi_m) \right| + |b_m| \left|I(\varphi^2_m)-Q_{N_r}(\varphi_m^2)\right|\right).
\]
Since $\varphi_m$ is analytic, the second term on the right hand side decays exponentially and the error is dominated by the first term. 

 Then for large $m$, $Q_{N_r}(\varphi_m^2) \asymp \|\phi_m\|_{L^2_r}^2 \asymp \frac{\overline{R}^2}{m}$. Together with \eqref{eq:Gauss_algebraic_rate}, we obtain
\BEA
|b_m - \widetilde{b}_m| \leq \frac{m}{\overline{R}^2} C_q N_r^{-q} \|g_0 \phi_m\|_{C^q}.\label{eq:bmerror}
\EEA
{With the assumption $f_0\in D_{s/2-1}$ and $D_{s/2-1}\subset H_r^s$, $f_0\in H_r^s$ and $g_0(r)=rf_0(r)\in H^s$,}  and the fact that $\phi_m \in C^\infty$, we have $g_0\phi_m \in C^{q}$, when $s>q+1/2$, by the Sobolev embedding theorem. That is,
\[
\|g_0\phi_m\|_{C^q} \leq C_s \|g_0\phi_m\|_{H^s}. % \leq C_{s} \left(\frac{\overline{R}}{m}\right)^s \|g_0\|_{H^s}\leq C_{\overline{R},s}\left(\frac{\overline{R}}{m}\right)^s \|f_0\|_{H^s}. 
\]
Since,
\[
\|g_0\phi_m\|_{L^2}^2 = \int_0^{\overline{R}} g_0^2(r)\phi_m^2(r)\,dr \leq \overline{R} \int_0^{\overline{R}} f_0^2(r)\phi_m^2(r) r\,dr = \overline{R} \|f_0 \phi_m\|_{L^2_r}^2.
\]
With similar argument, $\|g_0\phi_m\|_{H^s} \leq \overline{R}^{1/2}\|f_0\phi_m\|_{H^s_r}$. This implies that ,
{
\[
\|g_0\phi_m\|_{C^q} \leq C_s \overline{R}^{1/2}  \|f_0\phi_m\|_{H^s_r} \leq C_s\left(\frac{m}{\overline{R}}\right)^s \overline{R}^{1/2}\|f_0\|_{H^s_r}, 
\]}
where we have used $\|\partial^{\ell}_r\phi_m\|_{L^\infty}\leq \left(\frac{m}{\overline{R}}\right)^\ell$. Inserting this to \eqref{eq:bmerror}, the proof is complete.

\end{proof}

With these results, we deduce the following error estimate.

\begin{thm}\label{thm:coeff-error}
Let the assumptions in Theorem \ref{thm:basis-error} be valid. Let $\widetilde{f}_{0,M}$ be the discretization of $f_{0,M}$ defined as
\[\widetilde{f}_{0,M}=\sum_{m=1}^M \widetilde{b}_m\phi_m,\]
where $\widetilde{b}_m$ are
coefficients approximated as in Lemma~\ref{lem:error_bm} for all $m=1,\ldots, M$. Then, as $M\to \infty$,
{\BEA\|f_0-\widetilde{f}_{0,M}\|_{L^2_r}\leq C_1 \overline{R}^s M^{-s+\frac{1}{2}}+ C_{s,q} \overline{R}^{-(s+1/2)}M^{s+1}  N_r^{-q} \|f_0\|_{H^s_r},\label{thm4.2:errorf0}\EEA}
For the corresponding Poisson approximation, $\widetilde{p}_{1,M}=\sum_{m=1}^M \frac{\widetilde{b}_m}{\lambda_m}\phi_m$, we have
{
\BEA\|\widehat{p}_1-\widetilde{p}_{1,M}\|_{L^2_r}\leq C_2\overline{R}^{s+2}M^{-(s+\frac{3}{2})}+C_{s,q} \overline{R}^{-(s-3/2)}M^{s-1}  N_r^{-q} \|f_0\|_{H^s_r}.\label{thm4.2:errorp1}\EEA
}
\end{thm}

\begin{proof}
    By the triangle inequality,
\[
\|f_0-\widetilde{f}_{0,M}\|_{L^2_r}
\le \|f_0-f_{0,M}\|_{L^2_r}+\|f_{0,M}-\widetilde{f}_{0,M}\|_{L^2_r}.
\]
The first term is bounded by Theorem~\ref{thm:basis-error}. For the second term,
\[f_{0,M}-\widetilde{f}_{0,M}=\sum_{m=1}^M(b_m-\widetilde{b}_m)\phi_m.\]
Using orthogonality and Lemma~\ref{lem:error_bm},
{
\[\|f_{0,M}-\widetilde{f}_{0,M}\|_{L^2_r}^2=\sum_{m=1}^M |b_m-\widetilde{b}_m|^2 \|\phi_m\|_{L^2_r}^2 \leq C_{s,q}^2 \overline{R}^{-2(s+3/2)}N_r^{-2q} \|f_0\|^2_{H_r^s}\left(\sum_{m=1}^M m^{2(s+1)}\|\phi_m\|_{L^2_r}^2\right).\]}
Since $\|\phi_m\|_{L^2_r}^2 \asymp \frac{\overline{R}^2}{m}$, we have
{
\[\|f_{0,M}-\widetilde{f}_{0,M}\|_{L_r^2}\leq 
C_{s,q} \overline{R}^{-(s+1/2)}M^{s+1}  N_r^{-q} \|f_0\|_{H^s_r}, \]
where we use the fact that $\sum_{m=1}^M m^{2s+1}=\frac{M^{2s+2}}{2s+2}+O(M^{2s+1})$.}

Combining with the basis error in Theorem \ref{thm:basis-error}, we obtain \eqref{thm4.2:errorf0}.
For the corresponding solution $\widehat{p}_1$, using the scalings of $\lambda_m$ in \eqref{lambdaasymptotic} and the norm of the eigenbasis $\phi_m$,
we have
{
\[\|\widehat{p}_{1,M}-\widetilde{p}_{1,M}\|_{L^2_r}^2=\sum_{m=1}^M \frac{|b_m-\widetilde{b}_m|^2}{\lambda_m^2} \|\phi_m\|_{L^2_r}^2 \leq C^2_{s,q} \overline{R}^{-2(s+3/2)}  N_r^{-2q} \|f_0\|^2_{H^s_r}\left(\sum_{m=1}^M m^{2(s+1)} \frac{\|\phi_m\|_{L^2_r}^2}{\lambda_m^2} \right) .\]}
Using the scalings of $\lambda_m$ in \eqref{lambdaasymptotic} and the norm of the eigenbasis $\phi_m$, we have,
{
\[
\|\widehat{p}_{1,M}-\widetilde{p}_{1,M}\|_{L^2_r} \leq  C_{s,q} \overline{R}^{-(s-3/2)}M^{s-1}  N_r^{-q} \|f_0\|_{H^s_r},
\]
where we use the fact that $\sum_{m=1}^M m^{2s-3}=\frac{M^{2s-2}}{2s-2}
+O(M^{2s-3})
$.}

Combining with the basis error in Theorem \ref{thm:basis-error}, we obtain \eqref{thm4.2:errorp1}.
\end{proof}

\begin{remark}
Although this result is tailored to the Gauss-Legendre method that we used in our numerical implementation, the result can easily be modified for arbitrary quadrature rules. Additionally, we have several notes:
\begin{enumerate}
\item The result in Theorem 4.2 shows the projection error rate (first term) and the coefficient approximation error (second term). In our numerical experiments, we use a large number of quadrature points, $N_r$, to suppress the coefficient approximation error. 
The results in Figure~\ref{fig:r2_approx_error} correspond to the projection error rate, where one can see that the error in approximating $\widehat{p}_1$ decays $M^2$ faster than that of $f_0$.  
\item We should point out that while we focus on expansion of radial functions, one can also consider the Bessel expansion for nonradial functions. In fact, if the domain is arbitrary, the natural basis is to use eigenfunctions of the Laplace associated with the domain under the prescribed boundary conditions. {On a disk, these eigenfunctions are given by the Fourier-Bessel eigenfunctions, which provide a natural basis for approximating nonradial source terms. In Section \ref{sec:applications}, we present a numerical example based on the Fourier-Bessel expansion to illustrate the applicability of the proposed framework to nonradial source terms.}
\item In Theorem~\ref{thm:coeff-error}, the solution $\widehat{p}_{1,M}$ is defined with the exact $\lambda_m= \beta_m^2/\overline{R}^2$, where $\beta_m$ are zeros of the Bessel function $J_1$. Practically, we approximate these roots with the Newton scheme. While we believe that the error induced by this approximation is smaller than the one reported in Theorem~\ref{thm:coeff-error}, one can always add this error rate for completeness. 
\end{enumerate}
\end{remark}

\section{Applications and extensions}\label{sec:applications}
In this section, we demonstrate the applicability of the proposed method to a free-boundary tumor growth model which describes the interactions among tumor cells. We then demonstrate the extension of the proposed method to free boundary problems with more general source terms and present a numerical example with a nonradial source term.

\subsection{Application to tumor growth}

Let $\overline{\Omega}\subset\mathbb{R}^2$ denote a vascularized tissue region and let $\Omega(t)\subset\overline{\Omega}$ denote the tumor domain. We assume that tumor cells occupy the evolving domain $\Omega(t)$ and that their local proliferation rate is proportional to the nutrient concentration \cite{friedman2006bifurcation,chen2014tumor}. Since nutrients are supplied through the surrounding vascular network formed by angiogenesis, the nutrient concentration is modeled on the entire tissue domain $\overline{\Omega}$ rather than only within the tumor \cite{kerbel2008tumor,holash1999new}. Assuming that nutrient diffusion occurs on a much faster time scale than tumor growth, the nutrient concentration is governed by the quasi-steady reaction--diffusion equation
\begin{equation}\label{eqn:nutrient}
\left\{
\begin{array}{rcll}
-\Delta c &=& -c & \text{in }\, \overline{\Omega}, \\
c &=& c_0 & \text{on }\, \overline{\Gamma},
\end{array}
\right.
\end{equation}
where $c_0$ denotes the constant nutrient concentration supplied at the boundary. Cell proliferation induces a velocity field $\mathbf V$ describing the motion of tumor cells, which is assumed to obey Darcy's law \cite{byrne1995growth,greenspan1972models,greenspan1976growth,friedman2003hierarchy}. Under the assumption of constant cell density, conservation of mass together with Darcy's law yields
\begin{equation}\label{eqn:tumor}
\left\{
\begin{array}{rcll}
-\Delta p &=& \eta(c-\tilde{c}) & \text{in }\, \Omega(t), \\
p &=& \kappa & \text{on }\, \Gamma(t), \\
\frac{\partial p}{\partial \mathbf n} &=& -V_n & \text{on }\, \Gamma(t),
\end{array}
\right.
\end{equation}
where $p$ is the pressure of the fluid-like cells in $\Omega(t)$, $\eta$ is a positive parameter, and $\tilde{c}$ denotes the threshold nutrient concentration. 
In many continuum tumor growth models, the nutrient concentration is defined on the evolving tumor domain $\Omega(t)$ and is coupled with the free boundary through the moving interface \cite{lu2023bifurcation,lowengrub2010nonlinear}. Consequently, the nutrient field must be recomputed as the tumor boundary evolves, leading to a fully coupled moving-boundary problem. In this paper, our primary objective is to develop and validate an accurate numerical method for the generalized Hele--Shaw free boundary problem arising from the pressure equation in tumor growth, rather than for the coupled nutrient--tumor system. To this end, we prescribe the nutrient field by solving \eqref{eqn:nutrient} on a larger fixed tissue domain $\overline{\Omega}$, which is motivated by nutrient supply from the surrounding vascular network. The resulting nutrient distribution provides a spatially varying source term in the pressure equation \eqref{eqn:tumor}, allowing us to isolate the numerical treatment of the free boundary evolution while retaining a biologically meaningful forcing.

We consider the nutrient equation \eqref{eqn:nutrient} posed on a disk $\overline{\Omega}$ of radius $\overline{R}$. Due to the radial symmetry of $\overline{\Omega}$, the nutrient concentration admits the analytical solution
\[c(r)=c_0\frac{I_0(r)}{I_0(\overline{R})},\]
where $I_0$ denotes the modified Bessel function of the first kind of order zero. The corresponding source term in the free boundary equation \eqref{eqn:tumor} is then obtained from this nutrient distribution, $-\Delta p = \eta\left(c(r)-\tilde{c}\right)$. By decomposition, $p=p_1+p_2$, we reformulate \eqref{eqn:tumor} as
\begin{equation}\label{eqn:tumor_decompositon}
\begin{aligned}
-\Delta p_1=\eta\left(c(r)-\tilde{c}\right) \quad\text{on}~\overline{\Omega},
\qquad \text{and} \qquad
\begin{cases}
\begin{array}{rcll}
    -\Delta p_2 &=& 0,&\text{on}~\Omega(t),\\
p_2 & =& \kappa - \left.p_1\right|_{\Gamma},&\text{on}~\Gamma(t),\\
\frac{\partial p_2}{\partial \mathbf n}&=&-V_n-\left.\frac{\partial p_1}{\partial \mathbf n}\right|_{\Gamma},&\text{on}~\Gamma(t).
\end{array}
\end{cases}
\end{aligned}
\end{equation}

Combining with \eqref{eqn:nutrient}, we have $-\Delta (p_1+\eta c)=-\eta\tilde{c}$ which implies that $p_1(r) = -\eta c_0\frac{I_0(r)}{I_0(\overline{R})}+\frac{\eta\tilde{c}}{4}r^2$ is a particular solution. For the numerical experiments, we set $\overline{R}=3$, $c_0=1$, $\eta=10$, and $\tilde{c}=0.1$.

\begin{remark}
    Here we consider the nutrient equation \eqref{eqn:nutrient} on a fixed disk $\overline{\Omega}$ of radius $\overline{R}$. Due to the radial symmetry of $\overline{\Omega}$, the nutrient concentration admits the analytical solution, which allows the source term in the tumor model to be evaluated directly. This assumption is made solely for the purpose of validating the proposed framework.

    For a general tissue domain, an analytical solution is generally unavailable. However, the nutrient equation can be reformulated as a boundary integral equation using the Green's function denoted as $\widehat{G}(\mathbf x,\mathbf y)$ associated with the modified Helmholtz operator $(-\Delta+1)$. Specifically, \( \widehat{G}(\mathbf{x}, \mathbf{y}) =\frac{1}{2\pi} K_0\left(\|\mathbf{x} - \mathbf{y}\|\right) \) for two-dimensional cases, where $K_0$ is the modified Bessel function of the second kind of order zero. Applying the Green's third identity yields
\begin{equation}\label{eqn:BIE_nutrient}
       c(\mathbf x)=\int_{\overline{\Gamma}}\left(\widehat{G}(\mathbf x,\mathbf y)\frac{\partial c(\mathbf y)}{\partial \mathbf n(\mathbf y)}-c(\mathbf y)\frac{\partial\widehat{G}(\mathbf x,\mathbf y)}{\partial\mathbf n(\mathbf y)}\right)dS_{\mathbf y}, \quad \forall \mathbf x\in\overline{\Omega}
\end{equation}
Incorporating the jump condition and the boundary condition $c=c_0$ on $\overline{\Gamma}$, we have
\[\int_{\overline{\Gamma}}\widehat{G}(\mathbf x,\mathbf y)\frac{\partial c(\mathbf y)}{\partial \mathbf n(\mathbf y)}dS_{\mathbf y}=-\frac{c_0}{2}-\int_{\overline{\Gamma}}c_0\frac{\partial\widehat{G}(\mathbf x,\mathbf y)}{\partial\mathbf n(\mathbf y)}dS_{\mathbf y},\quad \mathbf x,\mathbf y\in\overline{\Gamma}\]
from which the unknown normal derivative $\left.\frac{\partial c}{\partial \mathbf n}\right|_{\overline{\Gamma}}$ can be determined. The nutrient concentration $c$ is then recovered from the boundary integral equation \eqref{eqn:BIE_nutrient} and the boundary condition. Therefore, the proposed framework naturally extends to general tissue domain without requiring volumetric discretization.
\end{remark}

\paragraph{Radial source term on a disk.} In this example, we consider the tumor region $\Omega(t)$ to be a disk whose initial boundary is a circle with radius 2,
\begin{equation}\label{eqn:tumor_IC1}
\mathbf x(0)=\Big(x_1(0),x_2(0)\Big)=\Big(2\cos(\theta),2\sin(\theta)\Big),\quad \theta\in[0,2\pi].
\end{equation}
For the circular case, the general solution of $-\Delta p_2=0$ is $p_2(r) =\kappa-\left.p_1\right|_{\Gamma}$ which implies that $V_n+\left.\frac{\partial p_1}{\partial \mathbf n}\right|_{\Gamma}=0$. Therefore, $\frac{dR(t)}{dt}=V_n=-\left.\frac{\partial p_1}{\partial\mathbf n}\right|_{\Gamma}=\eta c_0\frac{I_1(R(t))}{I_0(\overline{R})}-\frac{\eta\tilde{c}}{2}R(t)$, where $I_1$ denotes the modified Bessel function of the first kind of order one, and $R(t)$ denotes the radius of the evolving domain $\Omega(t)$. Numerically, we first approximate $f_0=f(r)-\overline{f}$ by the truncated eigenfunction expansion, where $f(r) = \eta\left(c(r)-\tilde{c}\right)$, and $\overline{f}=\frac{2}{\overline{R}^2}\int_0^{\overline{R}} f(r)rdr=\eta c_0\frac{2 I_1(\overline{R})}{\overline{R}I_0(\overline{R})}-\tilde{c}$ is the weighted radial mean. 
We employ Gauss–Legendre quadrature with $N_r=2000$ points and $M=200$ eigenfunctions. Equation \eqref{eqn:tumor_decompositon} is then solved using the proposed geometric local parameterization method with $N=400$, $\Delta t=10^{-5}$, and forward Euler time stepping. The resulting evolution of $\Omega(t)$ is shown in Figure \ref{fig:radial_disk-dynamics}, where snapshots at $t=0$, $0.05$, and $0.2$ are displayed. For validation, the ODE for $R(t)$ is solved using MATLAB's \texttt{ode45} solver as a reference solution. Figure \ref{fig:radial_disk-error} shows the absolute error in the radius over time.\begin{figure}[htbp]
\center
\subfigure[Interface evolution for a radial source term on a disk.]{\includegraphics[width=0.45\textwidth]{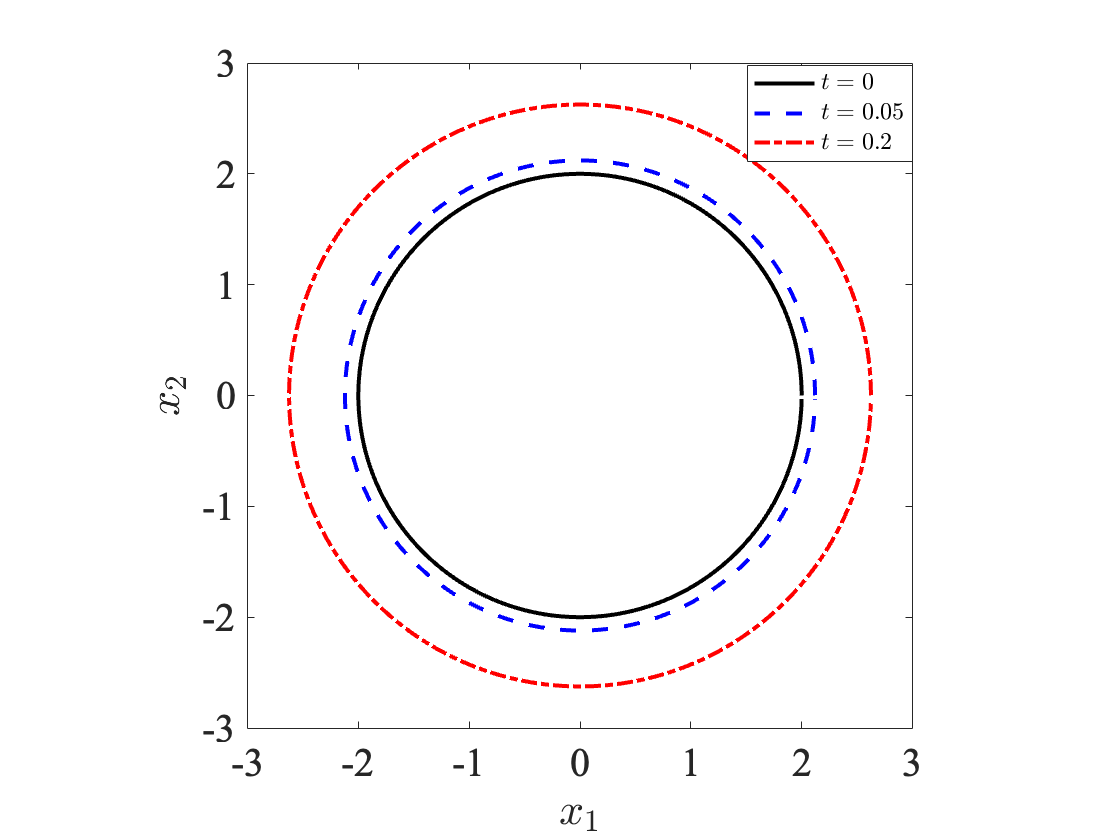}\label{fig:radial_disk-dynamics}}
\subfigure[Radius error over time.]{\includegraphics[width=0.45\textwidth]{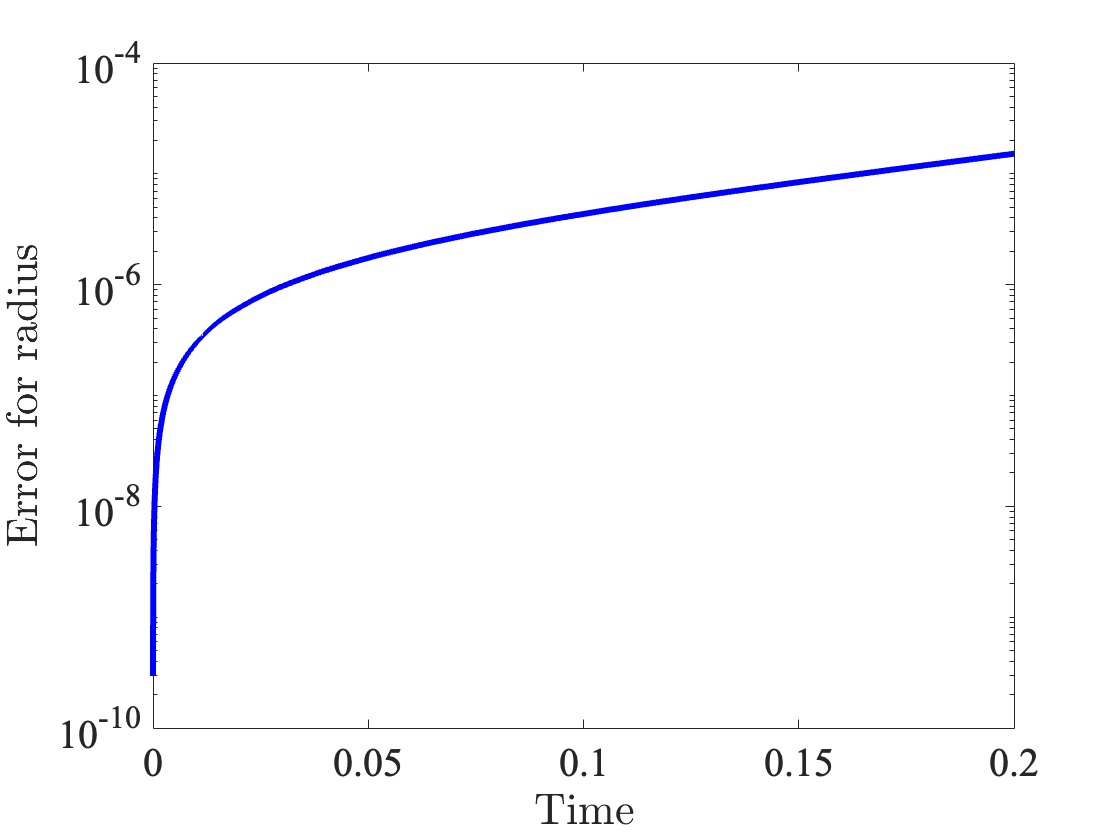}\label{fig:radial_disk-error}}
\caption{The evolution dynamics for a radial source term on a disk with initial radius 2. In (a),  the profiles of $\mathbf x(t)$ at $t=0$, $0.05$, $0.2$. In (b), The error for radius $R(t)$ is plotted at a time interval $[0,0.2]$. }
\label{fig:radial_disk}
\end{figure}

\paragraph{Radial source term on a perturbed disk.} Next, we consider a tumor region $\Omega(t)$ whose initial boundary is a perturbed disk given by,
\begin{equation}\label{eqn:perturbed-circle}
\begin{aligned}
 \mathbf x(0)&=\Big(x_1(0),x_2(0)\Big)=\Big(r(\theta)\cos(\theta),r(\theta)\sin(\theta)\Big),\quad   r(\theta) &=1+D_1\cos(D_2\theta), \theta\in[0,2\pi],
\end{aligned}
\end{equation}
where $D_1$ and $D_2$ are parameters. We set $D_1=0.1$ and $D_2=5$, while all other numerical parameters are kept the same as in the previous example.
As shown in Figure \ref{fig:radial_bessel}, the interface initially relaxes toward a circular shape due to surface tension effects. At later times, the forcing induced by the source term dominates, leading to overall tumor growth.

\begin{figure}[htbp]
\center
\subfigure[Interface evolution for a radial source term on a perturbed disk with $D_1=0.1$, $D_2=5$.]{\includegraphics[width=0.45\textwidth]{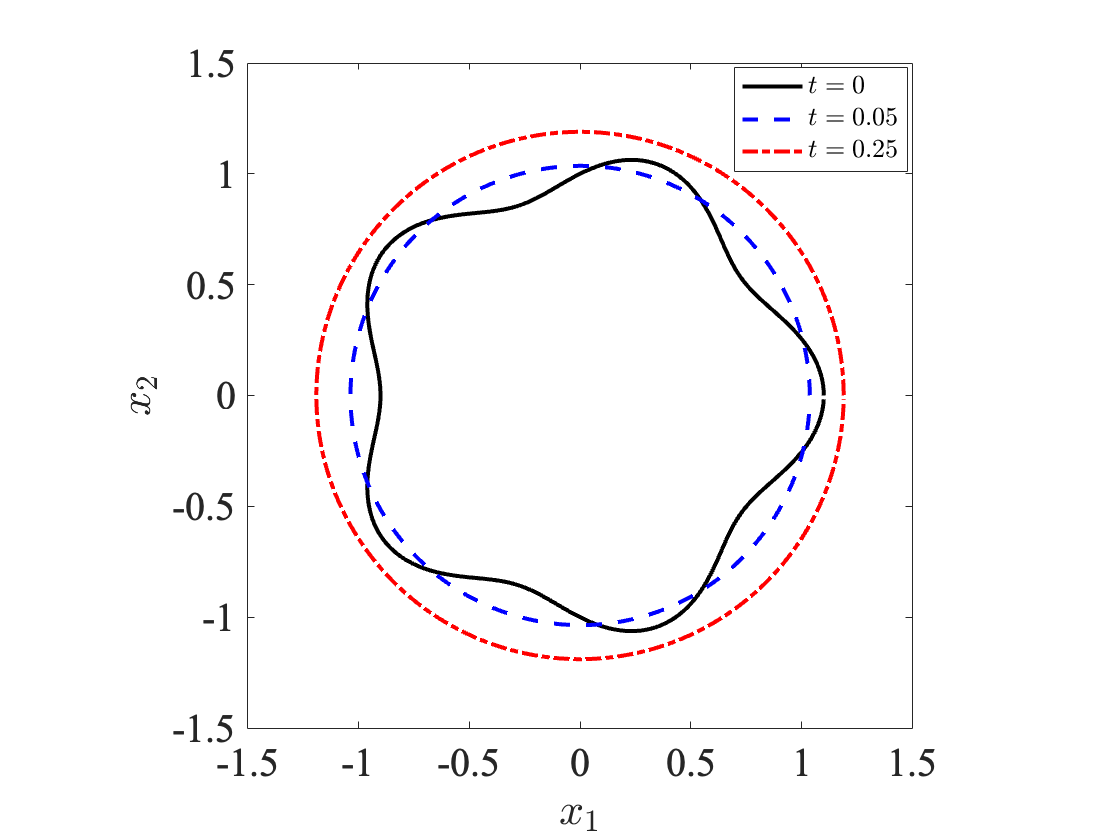}\label{fig:radial_bessel_dynamics}}
\subfigure[Evolution of then maximum and minimum distances from the boundary points to the center.]{\includegraphics[width=0.45\textwidth]{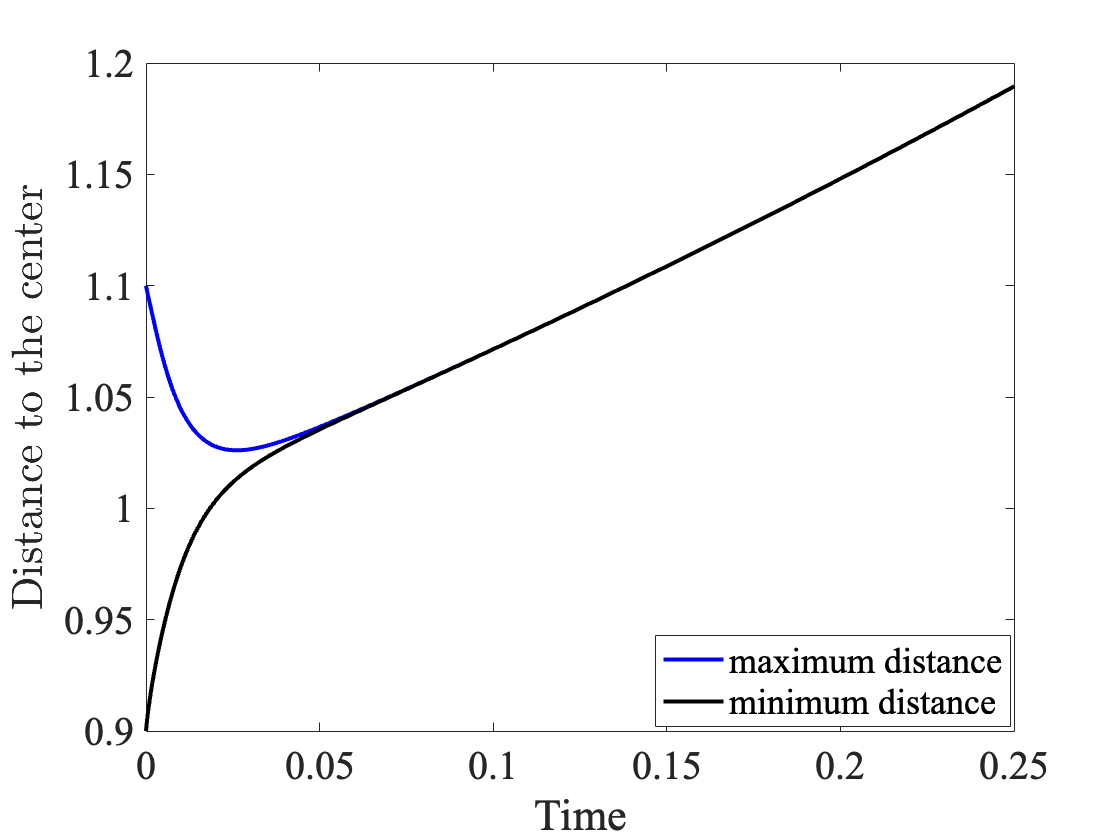}\label{radial_bessel_radius}}
\caption{The evolution dynamics for a radial source term on a perturbed circle with $D_1=0.1$ and $D_2=5$. In (a),  the profiles of $\mathbf x(t)$ at $t=0$, $0.05$, $0.25$ are shown. In (b), The maximum and minimum distances from the boundary points to the center are shown. }
\label{fig:radial_bessel}
\end{figure}

\subsection{Extension to nonradial source term} The numerical examples presented in the previous sections focus on radial source terms, for which only the radial Bessel eigenfunctions are required. We now demonstrate that the proposed method can be naturally extended to free boundary problems with nonradial source terms. The eigenfunction-based approximation is then constructed using the full Fourier-Bessel basis associated with the homogeneous Neumann eigenvalue problem for the negative Laplacian on the disk $\overline{\Omega}$ of radius $\overline{R}$. The eigenfunctions are given by
\[\psi_{1,nm}(r,\theta)=\cos(n\theta)J_n\left(\frac{\beta_{nm}r}{\overline{R}}\right),\quad \psi_{2,nm}(r,\theta)=\sin(n\theta)J_n\left(\frac{\beta_{nm}r}{\overline{R}}\right),\]
where $J_n$ is the Bessel function of the first kind of order $n$, and $\beta_{nm}$ is the $m$-th positive zero of $J_n'$. {The positive zeros $\beta_{nm}$ are computed numerically by first locating intervals containing successive roots through sign changes of $J_n'$, followed by refinement using MATLAB's \texttt{fzero} routine.}

For a general source function $f(r,\theta)$, as in Section \ref{subsec:eigenfunction}, we decompose $f=f_0+\overline{f}$, where $\overline{f} =\frac{1}{\pi\overline{R}^2} \int_0^{2\pi}\int_0^{\overline{R}}f(r,\theta)rdrd\theta$, so that $f_0$ belongs to the mean-zero subspace spanned by the Neumann eigenfunctions. We then approximate $f_0$ by the truncated eigenfunction expansion
\[f_0(r,\theta)\approx f_{0,N_{\theta},M}(r,\theta)=\sum_{n=0}^{N_\theta}\sum_{m=1}^M \left(A_{nm}\cos(n\theta)+B_{nm}\sin(n\theta)\right)J_n\left(\frac{\beta_{nm}r}{\overline{R}}\right),\]
where $N_\theta$ is the maximal number of angular modes, and $M$ is the maximal number of radial modes. The coefficients are given by the corresponding weighted $L^2$ projections. For $n\geq 1$,
\[A_{nm}=\frac{\int_0^{2\pi}\int_0^{\overline{R}} f_0(r,\theta)J_n\left(\frac{\beta_{nm}r}{\overline{R}}\right)\cos(n\theta)rdrd\theta}{\int_0^{2\pi}\int_0^{\overline{R}} \left(J_n\left(\frac{\beta_{nm}r}{\overline{R}}\right)\cos(n\theta)\right)^2rdrd\theta},\quad B_{nm}=\frac{\int_0^{2\pi}\int_0^{\overline{R}} f_0(r,\theta)J_n\left(\frac{\beta_{nm}r}{\overline{R}}\right)\sin(n\theta)rdrd\theta}{\int_0^{2\pi}\int_0^{\overline{R} }\left(J_n\left(\frac{\beta_{nm}r}{\overline{R}}\right)\sin(n\theta)\right)^2rdrd\theta}.\]
Then the approximation of the corresponding solution $\widehat{p}_1$ follows directly from the eigenvalue relation
\[\widehat{p}_1\approx\widehat{p}_{1,N_{\theta},M}=\sum_{n=0}^{N_\theta}\sum_{m=1}^M \left(\frac{A_{nm}}{\lambda_{nm}}\cos(n\theta)+\frac{B_{nm}}{\lambda_{nm}}\sin(n\theta)\right)J_n\left(\frac{\beta_{nm}r}{\overline{R}}\right),\quad \lambda_{nm}=\left(\frac{\beta_{nm}}{\overline{R}}\right)^2.\]
We illustrate this extension with the following numerical example, \[f(r,\theta)=r^2e^{\frac{r\cos(\theta)}{2\overline{R}}},\quad r\in[0,\overline{R}],~\theta\in[0,2\pi].\] 
We set $\overline{R}=3$ and apply the proposed eigenfunction approximation using Gauss-Legendre quadrature with $N_r=2000$ quadrature points in the radial direction and $N_a=256$ uniformly distributed points in the angular direction. We first investigate the convergence of the proposed approximation with respect to the number of radial modes. The number of angular modes is fixed at $N_{\theta}=10$, while the number of radial modes is taken as $M=100,200,400$, and $800$. As shown in Figure \ref{fig:radial_convergence}, the approximation error $\|e_f\|_{L^2}=\left(\int_0^{2\pi}\int_0^{\overline{R}} \left(f(r,\theta)-f_M(r,\theta)\right)^2rdrd\theta\right)^{\frac{1}{2}}$ decays approximately as $M^{-1.5}$. Next, we examine the convergence with respect to the number of angular modes. The number of radial modes is fixed at $M=1600$, while the number of angular modes is varied as $N_{\theta}=1,2,3$, and $4$. Figure \ref{fig:angular_convergence} demonstrates exponential convergence with increasing $N_{\theta}$.

\begin{figure}[htbp]
\center
\subfigure[]{\includegraphics[width=0.45\textwidth]{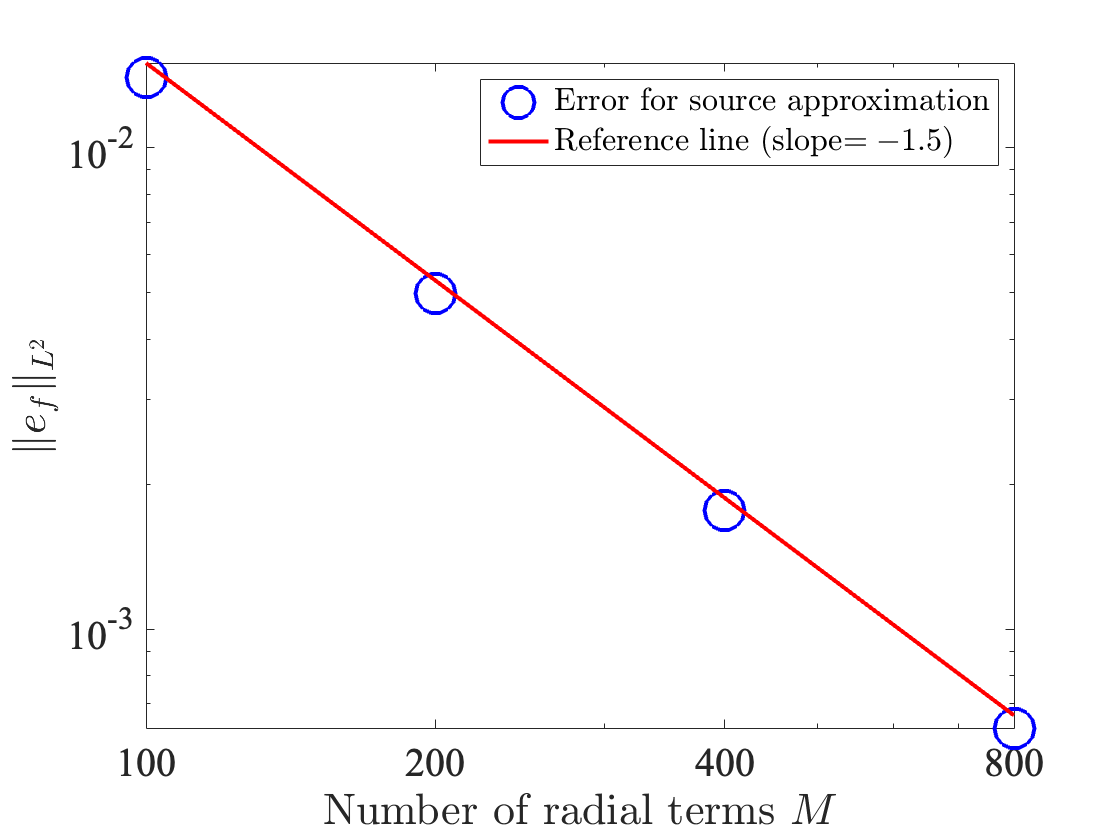}
\label{fig:radial_convergence}}
\subfigure[]{\includegraphics[width=0.45\linewidth]{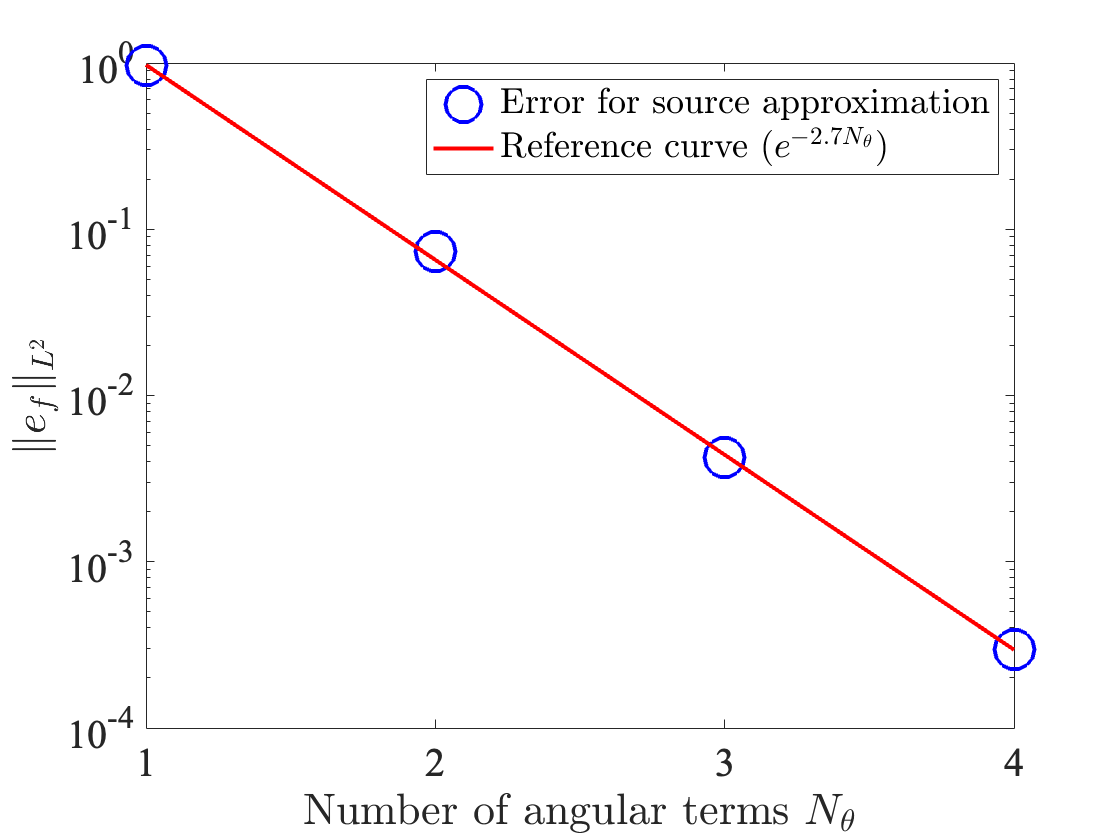}
\label{fig:angular_convergence}}
\caption{Weighted $L^2$ errors for the approximation of the source function $f$.}
\label{fig:nonradial_approx_error}
\end{figure}

Having verified the convergence of the proposed eigenfunction approximation, we next solve the free boundary problem with the nonradial source term,
\begin{equation}\label{eqn:nonradial}
\begin{aligned}
-\Delta p_1=f(r,\theta) \quad\text{on}~\overline{\Omega},
\qquad \text{and} \qquad
\begin{cases}
\begin{array}{rcll}
    -\Delta p_2 &=& 0,&\text{on}~\Omega(t),\\
p_2 & =& \kappa - \left.p_1\right|_{\Gamma},&\text{on}~\Gamma(t),\\
\frac{\partial p_2}{\partial \mathbf n}&=&-V_n-\left.\frac{\partial p_1}{\partial \mathbf n}\right|_{\Gamma},&\text{on}~\Gamma(t),
\end{array}
\end{cases}
\end{aligned}
\end{equation}
where the initial boundary of $\Omega(t)$ is a perturbed disk given by,
\begin{equation}\label{nonradial_IC}
\begin{aligned}
    r(\theta) &=2+D_1\cos(D_2\theta),\\
\mathbf x(0)&=\Big(x_1(0),x_2(0)\Big)=\Big(r(\theta)\cos(\theta),r(\theta)\sin(\theta)\Big),\quad \theta\in[0,2\pi],
\end{aligned}
\end{equation}
where $D_1$ and $D_2$ are parameters with values $D_1=0.1$ and $D_2=5$. We apply the proposed eigenfunction approximation using $M=600$ radial terms and $N_{\theta}=6$ angular terms. \eqref{eqn:nonradial} is then solved by the proposed geometric local parameterization method with $N=400$, $\Delta t=10^{-5}$, and the Forward Euler method. The evolution dynamics of $\Omega(t)$ is shown in Figure \ref{fig:nonradial_dynamics}, while Figure \ref{fig:nonradial_radius} depicts the evolution of then maximum and minimum distances from the boundary points to the center of mass.
\begin{figure}[htbp]
\center
\subfigure[Interface evolution for a nonradial source term on a perturbed disk with $D_1=0.1$, $D_2=5$.]{\includegraphics[width=0.45\textwidth]{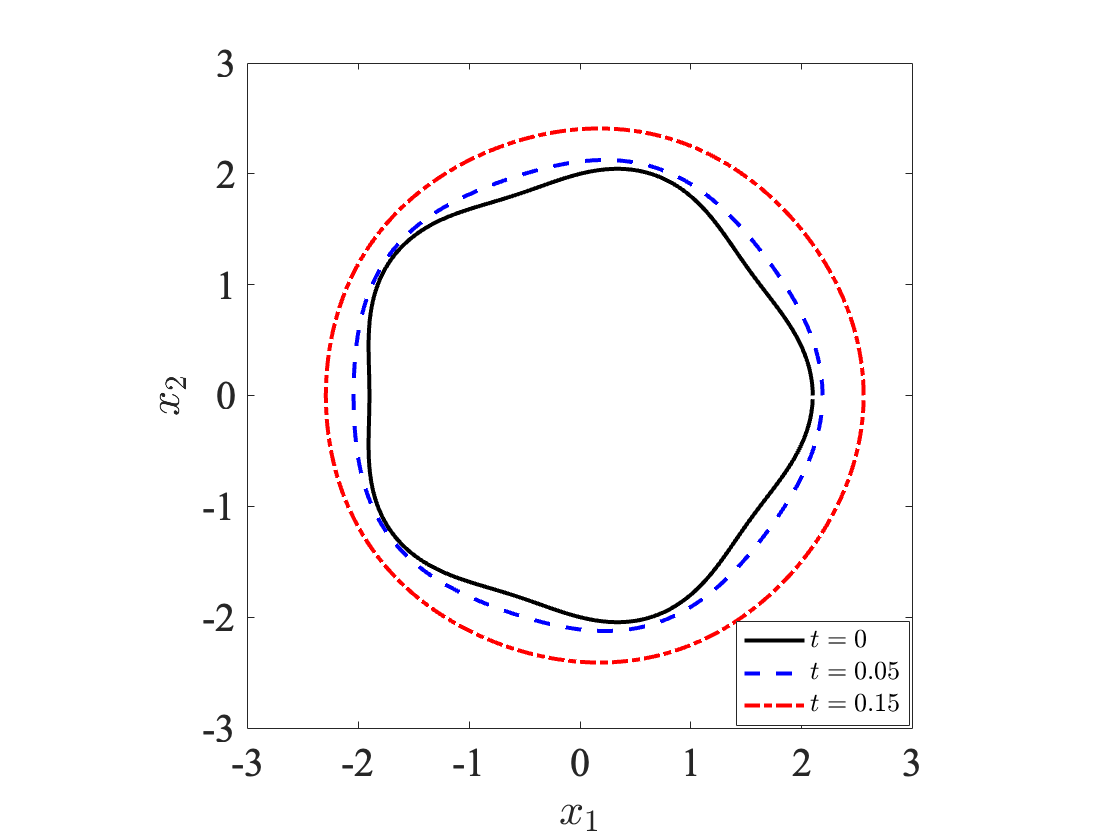}\label{fig:nonradial_dynamics}}
\subfigure[Evolution of the maximum and minimum distances from the boundary points to the center of mass.]{\includegraphics[width=0.45\textwidth]{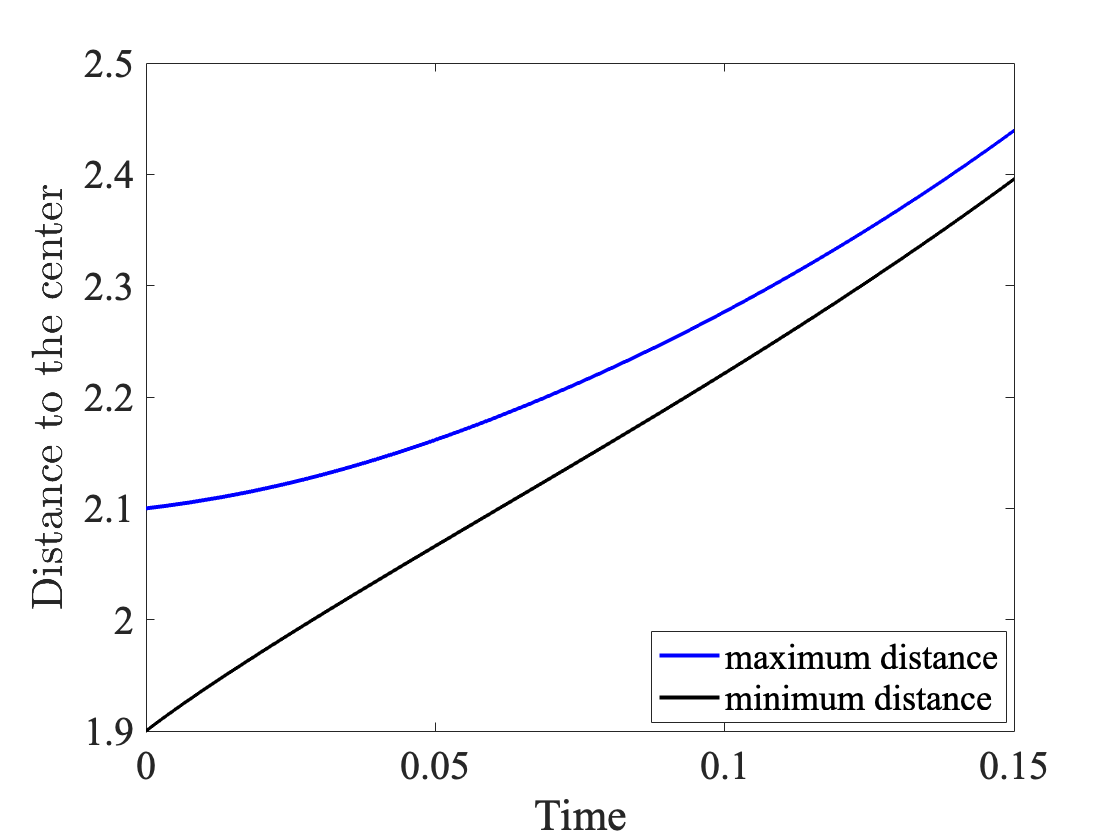}\label{fig:nonradial_radius}}
\caption{The evolution dynamics for a nonradial source term on a perturbed circle with $D_1=0.1$ and $D_2=5$. In (a),  the profiles of $\mathbf x(t)$ at $t=0$, $0.05$, $0.15$ are shown. In (b), The maximum and minimum distances from the boundary points to the center of mass are shown. }
\label{fig:nonradial}
\end{figure}

\section{Conclusion}

In this work, we have presented a meshfree numerical framework for solving Hele-Shaw free boundary problems with surface tension and source terms. By combining geometric local parameterization with boundary integral methods, the governing equations are reformulated as a boundary-only system, eliminating the need for volumetric discretization of the evolving domain. To efficiently treat general source terms, we introduced an eigenfunction-based approximation on a fixed auxiliary domain and established rigorous error estimates for both the basis truncation and coefficient approximation, providing a theoretical foundation for the proposed approach.

The numerical results demonstrate that the proposed method accurately captures the free-boundary evolution while achieving the predicted convergence rates. In particular, the approximation error for the particular solution is observed to decay faster than that of the source function. The proposed framework is further applied to a coupled free boundary tumor growth model, demonstrating its flexibility in handling coupled moving-boundary problems without remeshing.

The proposed approach naturally extends to more general source terms and geometries. In particular, future work includes the construction of Laplacian eigenfunction approximations on arbitrary domains, the extension to three-dimensional free boundary problems, and applications to more complex biological and multiphysics systems.

\section*{Acknowledgments}
ZZ was supported by the ICDS seed grant and National Institute of General Medical Sciences (NIGMS) grant 1R35GM146894.
JH was partially supported by NSF grant DMS-2505605, WH was supported by NIGMS 1R35GM146894 and the Huck Chair in AI Mathematical Modeling from Penn State University's Huck Institutes of the Life Sciences.

\section*{Conflicts of interest / Competing interests
}
The authors declare that they have no competing interests.

\section*{Availability of data and materials}
No datasets were generated or analyzed during the current study. Data sharing is not applicable to this article.

%\bibliographystyle{plain}  
%\bibliography{references}

\begin{thebibliography}{10}

\bibitem{bowman2012introduction}
Frank Bowman.
\newblock {\em Introduction to Bessel functions}.
\newblock Courier Corporation, 2012.

\bibitem{byrne1995growth}
HM~Byrne and M~A\_~J Chaplain.
\newblock Growth of nonnecrotic tumors in the presence and absence of
  inhibitors.
\newblock {\em Mathematical biosciences}, 130(2):151--181, 1995.

\bibitem{chen2003free}
X.~Chen and A.~Friedman.
\newblock A free boundary problem for an elliptic-hyperbolic system: an
  application to tumor growth.
\newblock {\em SIAM J. Math. Anal.}, 35(4):974--986, 2003.

\bibitem{chen2014tumor}
Ying Chen and John~S Lowengrub.
\newblock Tumor growth in complex, evolving microenvironmental geometries: a
  diffuse domain approach.
\newblock {\em Journal of theoretical biology}, 361:14--30, 2014.

\bibitem{courant2024methods}
Richard Courant and David Hilbert.
\newblock {\em Methods of mathematical physics, volume 2}.
\newblock John Wiley \& Sons, 2024.

\bibitem{feng2004analysis}
Xiaobing Feng and Andreas Prohl.
\newblock Analysis of a fully discrete finite element method for the phase
  field model and approximation of its sharp interface limits.
\newblock {\em Mathematics of computation}, 73(246):541--567, 2004.

\bibitem{friedman2012variational}
A.~Friedman and J.~Spruck.
\newblock {\em Variational and free boundary problems}, volume~53.
\newblock Springer Science \& Business Media, 2012.

\bibitem{friedman1963generalized}
Avner Friedman.
\newblock {\em Generalized functions and partial differential equations}.
\newblock Prentice Hall, 1963.

\bibitem{friedman2000free}
Avner Friedman.
\newblock Free boundary problems in science and technology.
\newblock {\em Notices of the AMS}, 47(8):854--861, 2000.

\bibitem{friedman2003hierarchy}
Avner Friedman.
\newblock A hierarchy of cancer models and their mathematical challenges.
\newblock {\em Discrete and Continuous Dynamical Systems-B}, 4(1):147--159,
  2003.

\bibitem{friedman2015free}
Avner Friedman.
\newblock Free boundary problems in biology.
\newblock {\em Philosophical Transactions of the Royal Society A: Mathematical,
  Physical and Engineering Sciences}, 373(2050):20140368, 2015.

\bibitem{friedman2006bifurcation}
Avner Friedman and Bei Hu.
\newblock Bifurcation from stability to instability for a free boundary problem
  arising in a tumor model.
\newblock {\em Archive for rational mechanics and analysis}, 180(2):293--330,
  2006.

\bibitem{greenspan1972models}
Harvey~P Greenspan.
\newblock Models for the growth of a solid tumor by diffusion.
\newblock {\em Studies in Applied Mathematics}, 51(4):317--340, 1972.

\bibitem{greenspan1976growth}
Harvey~P Greenspan.
\newblock On the growth and stability of cell cultures and solid tumors.
\newblock {\em Journal of theoretical biology}, 56(1):229--242, 1976.

\bibitem{HHHS}
W.~Hao, J.~Hauenstein, B.~Hu, and A.~Sommese.
\newblock A three-dimensional steady-state tumor system.
\newblock {\em Appl. Math. Comput.}, 218(6):2661--2669, 2011.

\bibitem{HHLS}
Wenrui Hao, Bei Hu, Shuwang Li, and Lingyu Song.
\newblock Convergence of boundary integral method for a free boundary system.
\newblock {\em J. Comput. Appl. Math.}, 334:128--157, 2018.

\bibitem{harlim2023radial}
John Harlim, Shixiao~Willing Jiang, and John~Wilson Peoples.
\newblock Radial basis approximation of tensor fields on manifolds: from
  operator estimation to manifold learning.
\newblock {\em J. Mach. Learn. Res.}, 24(345):1--85, 2023.

\bibitem{HSH}
H.~Hele-Shaw.
\newblock Flow of water.
\newblock {\em Nature}, 58:520, 1898.

\bibitem{holash1999new}
J~Holash, SJ~Wiegand, and GD~Yancopoulos.
\newblock New model of tumor angiogenesis: dynamic balance between vessel
  regression and growth mediated by angiopoietins and vegf.
\newblock {\em Oncogene}, 18(38):5356--5362, 1999.

\bibitem{hou1994removing}
Thomas~Y Hou, John~S Lowengrub, and Michael~J Shelley.
\newblock Removing the stiffness from interfacial flows with surface tension.
\newblock {\em Journal of Computational Physics}, 114(2):312--338, 1994.

\bibitem{jiang2024generalized}
Shixiao~Willing Jiang, Rongji Li, Qile Yan, and John Harlim.
\newblock Generalized finite difference method on unknown manifolds.
\newblock {\em J. Comput. Phys.}, 502:112812, 2024.

\bibitem{kerbel2008tumor}
Robert~S Kerbel.
\newblock Tumor angiogenesis.
\newblock {\em New England Journal of Medicine}, 358(19):2039--2049, 2008.

\bibitem{kress1989linear}
Rainer Kress.
\newblock {\em Linear integral equations}, volume~82.
\newblock Springer, 1989.

\bibitem{lowengrub2010nonlinear}
John~S Lowengrub, Hermann~B Frieboes, Fang Jin, Yao-Li Chuang, Xiangrong Li,
  Paul Macklin, Steven~M Wise, and Vittorio Cristini.
\newblock Nonlinear modelling of cancer: bridging the gap between cells and
  tumours.
\newblock {\em Nonlinearity}, 23(1):R1--R91, 2010.

\bibitem{lu2023bifurcation}
Min-Jhe Lu, Wenrui Hao, Bei Hu, and Shuwang Li.
\newblock Bifurcation analysis of a free boundary model of vascular tumor
  growth with a necrotic core and chemotaxis.
\newblock {\em Journal of mathematical biology}, 86(1):19, 2023.

\bibitem{macklin2007nonlinear}
Paul Macklin and John Lowengrub.
\newblock Nonlinear simulation of the effect of microenvironment on tumor
  growth.
\newblock {\em Journal of theoretical biology}, 245(4):677--704, 2007.

\bibitem{mirzaei2012generalized}
Davoud Mirzaei, Robert Schaback, and Mehdi Dehghan.
\newblock On generalized moving least squares and diffuse derivatives.
\newblock {\em IMA J. Numer. Anal.}, 32(3):983--1000, 2012.

\bibitem{olver2010nist}
Frank W.~J. Olver, Daniel~W. Lozier, Ronald~F. Boisvert, and Charles~W. Clark,
  editors.
\newblock {\em NIST Handbook of Mathematical Functions}.
\newblock Cambridge University Press, 2010.

\bibitem{pozrikidis1992boundary}
Constantine Pozrikidis.
\newblock {\em Boundary integral and singularity methods for linearized viscous
  flow}.
\newblock Cambridge university press, 1992.

\bibitem{richardson1972hele}
Stanley Richardson.
\newblock Hele shaw flows with a free boundary produced by the injection of
  fluid into a narrow channel.
\newblock {\em Journal of Fluid Mechanics}, 56(4):609--618, 1972.

\bibitem{ryskin1984numerical}
G~Ryskin and LG~Leal.
\newblock Numerical solution of free-boundary problems in fluid mechanics. part
  1. the finite-difference technique.
\newblock {\em Journal of fluid mechanics}, 148:1--17, 1984.

\bibitem{sackinger1996newton}
Phillip~A Sackinger, Peter~Randall Schunk, and Rekha~R Rao.
\newblock A newton--raphson pseudo-solid domain mapping technique for free and
  moving boundary problems: a finite element implementation.
\newblock {\em Journal of computational physics}, 125(1):83--103, 1996.

\bibitem{saffman1958penetration}
Philip~Geoffrey Saffman and Geoffrey~Ingram Taylor.
\newblock The penetration of a fluid into a porous medium or hele-shaw cell
  containing a more viscous liquid.
\newblock {\em Proceedings of the Royal Society of London. Series A.
  Mathematical and Physical Sciences}, 245(1242):312--329, 1958.

\bibitem{shen1995efficient}
Jie Shen.
\newblock Efficient spectral-galerkin method ii. direct solvers of second-and
  fourth-order equations using chebyshev polynomials.
\newblock {\em SIAM Journal on Scientific Computing}, 16(1):74--87, 1995.

\bibitem{tryggvason1983numerical}
Gretar Tryggvason and Hassan Aref.
\newblock Numerical experiments on hele shaw flow with a sharp interface.
\newblock {\em Journal of Fluid Mechanics}, 136:1--30, 1983.

\bibitem{zhang2026geometric}
Zengyan Zhang, Wenrui Hao, and John Harlim.
\newblock Geometric local parameterization for solving hele-shaw problems with
  surface tension.
\newblock {\em SIAM J. Sci. Comput. (to appear)}, 2026.

\end{thebibliography}

\appendix
\section{Generalized Moving Least Squares Approximation for the Local Parametrization}\label{appx:GMLS}
In this appendix, we review the Generalized Moving Least Squares (GMLS) method and its application to approximate a local parameterization of smooth manifolds, which subsequently yields a higher-order approximation of the local tangent space.

We first employ a classical local singular value decomposition (SVD) method based on neighboring point cloud data to approximate the local tangent space at each point on the manifold $\Gamma\subset\mathbb{R}^d$,. More specifically, for each point \(\mathbf x_i\) in the point cloud $X\subset \mathbb{R}^n$, we construct a local distance matrix $\mathbf D_{i}=[\mathbf D_{i_1},\dots,\mathbf D_{i_k}]\in\mathbb {R}^{n\times k}$, where $k>d$ and $\mathbf D_{i_j}:=\mathbf x_{i_j}-\mathbf x_i$, $j=1,\dots,k$, using its \(k\)-nearest neighbors. The approximated tangent space $\widetilde{T_{\mathbf x_i}\Gamma}$ is then obtained by taking the SVD of $\mathbf D_i$, yielding an orthonormal basis $\{\tilde{\mathbf t}_i^{(1)},\dots,\tilde{\mathbf t}_i^{(d)}\}$. We should point out that, for uniformly sampled random data of size $N$, the SVD approximation converges at order $N^{-1/d}$, provided that a sufficiently large number of $k$-nearest points is used (see Remark 9 in \cite{harlim2023radial}). 

To attain a higher-order approximation of the local tangent space, we consider the Generalized Moving Least Square method. Generally, we use an intrinsic polynomial to approximate smooth function $g:\Gamma\to\mathbb{R}$ over the neighborhood of $\mathbf x_i$ which is the optimal solution for the following least-squares problem,
\[\min_{\hat{g}\in\mathbb{P}_{\mathbf x_i}^{l,d}}\sum_{j=1}^k\left(g(\mathbf x_{i_j})-\hat{g}(\mathbf x_{i_j})\right)^2,\]
where $\mathbb{P}_{\mathbf x_i}^{l,d}$ is the space of intrinsic polynomial with degree up to $l$ at the point $\mathbf x_i\in\Gamma\subset\mathbf R^d$. We refer to $\hat{g}$ as the Generalized Moving Least Squares (GMLS) estimate of $g$.

In the tangent space approximation, the function $g$ will be the local parameterization of the manifold, with the tangent space represented by its Jacobian. In this work, we restrict our attention to the case $d=1$ and $n=2$. Let $\tilde{\mathbf t}_i$ and $\tilde{\mathbf n}_i$ denote the approximate tangent and normal vectors obtained from the local SVD method. The GMLS approximation for the local parametrization over a neighborhood of each base point $\mathbf x_i$ is stated as follows,

\begin{enumerate}
    \item Obtain the intrinsic polynomial $\tilde p_i:\widetilde{T_{\mathbf x_i}\Gamma}\rightarrow\mathbb{R}$ through a least-square fitting procedure,
\[\tilde p_i(\tilde s)=\tilde\alpha_{i,1} \tilde s+\tilde\alpha_{i,2} \tilde s^2+\cdots+\tilde\alpha_{i,\ell} \tilde s^\ell,\]
where $\tilde s\in \{\tilde{\mathbf t}_i\cdot(\mathbf x_{i_1}-\mathbf x_{i_1}),\dots,\tilde{\mathbf t}_i\cdot(\mathbf x_{i_k}-\mathbf x_{i_1})\}\subset \widetilde{T_{\mathbf x_i}\Gamma}$, and the coefficients are obtained by a least squares fit, i.e., 
\[(\tilde\alpha_{i,1},\dots,\tilde\alpha_{i,\ell})=\arg\min \sum_{j=1}^k\Big(\tilde p_i(\tilde s)-\tilde{\mathbf n}_{i}\cdot(\mathbf x_{i_j}-\mathbf x_{i_1})\Big)^2.\]

This defines a local coordinate chart for the manifold near $\mathbf x_i$ using the embedding map $\tilde\iota_i:\widetilde{T_{\mathbf x_i}\Gamma}\rightarrow\mathbb{R}^2$,
\begin{equation}\label{svd-map}
\tilde\iota_i(\tilde s)=\mathbf x_i+\tilde{\mathbf t}_i \tilde s+\tilde{\mathbf n}_i \tilde p_i(\tilde s).
\end{equation}

    \item Refine the local tangent vectors approximation using the Jacobian of $\tilde{\iota}_i$ defined in \eqref{svd-map}. The GMLS approximation of the tangent vectors over the neighborhood of $\mathbf x_i$ is
    \begin{equation}\label{gmls_tangent}
\mathbf{\hat{t}}_i(\tilde s)=\tilde{\mathbf t}_i+\tilde{\mathbf n}_i\tilde p'_i(\tilde s),
\end{equation}
with the estimated tangent vector at $\mathbf x_i$ denoted by $\mathbf{\hat{t}}_i=\mathbf{\hat{t}}_i(0)$, as a GMLS approximation to the underlying tangent vector $\mathbf t(\mathbf x_i)$. The corresponding unit normal vector $\hat{\mathbf n}_i$ is constructed such that $\hat{\mathbf t}_i\cdot\hat{\mathbf n}_i=0$ with $\hat{\mathbf n}_i$ oriented outward from $\Omega(t)$. The GMLS fitting procedure is subsequently repeated using the updated local coordinates $(\hat{\mathbf t}_i,\hat{\mathbf n}_i)$, yielding the refined local embedding map $\hat{\iota}_i:T_{\mathbf x_i}\Gamma\rightarrow \mathbb{R}^2$,
\begin{equation}\label{appx:gmls-map}
\hat{\iota}_i(s)=\mathbf x_i+\mathbf{\hat{t}}_is+\mathbf{\hat{n}}_i p_i(s),
\end{equation}
where $p_i(s)=\alpha_{i,1}s+\alpha_{i,2} s^2+\cdots+\alpha_{i,\ell} s^\ell$ and $\alpha_{i,1},\alpha_{i,2},\cdots,\alpha_{i,\ell}$ are obtained via GMLS fitting using the data pairs from $s\in T(i)=\{\mathbf{\hat{t}}_i\cdot(\mathbf x_{i_1}-\mathbf x_{i_1}),\dots,\mathbf{\hat{t}}_i\cdot(\mathbf x_{i_k}-\mathbf x_{i_1})\}\subset T_{\mathbf x_i}\Gamma$ and $N(i) =\{\mathbf{\hat{n}}_i\cdot(\mathbf x_{i_1}-\mathbf x_{i_1}),\dots,\mathbf{\hat{n}}_i\cdot(\mathbf x_{i_k}-\mathbf x_{i_1})\}$.
\end{enumerate}
\begin{remark} (Iterative refinement). 
   The additional normal correction term in \eqref{appx:gmls-map} yields a refinement of the estimation for the local parameterization by applying the GMLS procedure iteratively on the local coordinates $(\hat{\mathbf t}_i, \hat{\mathbf n}_i)$. We iterate this process multiple times, updating tangent vector estimate with the Jacobian of $\hat{\iota}_i(0)$, and stop when the polynomial linear coefficient $p'_i(0)$ falls below a desired tolerance as a stopping criterion, as further iterations provide negligible improvement. In our numerical simulations, we take this stopping criterion to be  $p'_i(0)=\alpha_{i,1} \approx 10^{-12}$, and denote the resulting estimated local tangent and normal vectors as $(\mathbf{\hat{t}}_i,\mathbf{\hat{n}}_i)$ for notational simplicity. As a result, $p_i(s)$ effectively has no linear term, i.e., $p_i(s)=\alpha_{i,2} s^2+\cdots+\alpha_{i,\ell} s^\ell$.
\end{remark}

\begin{remark} (Order of convergence).
    The GMLS tangent vector estimate $\hat{\mathbf t}_i$ converges at rate $(N^{-1}\log N)^{\ell}$, which indicates an improved estimate for $\ell>1$ compared to the SVD approximation, $\tilde{\mathbf t}_i$, and the GMLS approximation of the curvature converges at rate $(N^{-1}\log N)^{\ell-1}$ for uniformly sampled random data of size $N$. For further details of the convergence results, we refer to the previous work \cite{zhang2026geometric}.
\end{remark}

\end{document}